\documentclass[oneside, a4paper, 11pt]{amsart}
\usepackage[utf8]{inputenc}
\usepackage[T1]{fontenc}
\usepackage{amsmath,amsthm,amsrefs,amssymb,bbm,comment,mathrsfs,hyperref}

\calclayout

\pdfoutput=1
\usepackage{graphicx,enumitem,stmaryrd}

\usepackage{anyfontsize}
\allowdisplaybreaks[4]

\usepackage[table]{xcolor}

\usepackage{adjustbox}

\newcommand{\Z}{\mathbb{Z}}

\newcommand{\cB}{\mathcal{B}}
\newcommand{\cC}{\mathcal{C}}

\newtheorem{theorem}{Theorem}
\newtheorem{Theox}{Theorem}

\newtheorem{corollary}[theorem]{Corollary}

\newtheorem{remark}[theorem]{Remark}
\newtheorem{definition}[theorem]{Definition }
\newtheorem{example}[theorem]{Example}
\newtheorem{lemma}[theorem]{Lemma}
\newtheorem{proposition}[theorem]{Proposition}

\newcommand{\Ker}[1]{\mathrm{Ker}(#1)}
\newcommand{\asym}{\widehat{A}}
\newcommand{\cT}{\mathcal{T}}
\newcommand{\ann}{\mathrm{ann}}
\newcommand{\Prim}{\mathrm{Prim}}
\newcommand{\rank}{\mathrm{rank}}
\newcommand{\MaxSpec}{\mathrm{MaxSpec}}
\newcommand{\rad}{\mathrm{rad}}
\newtheorem{note}{Remark}[section]
\newcommand{\tld}[1]{\widetilde{#1}}

\newcommand{\br}[1]{\overline{#1}}
\newcommand{\End}{\mathrm{End}}
\newcommand{\Hom}{\mathrm{Hom}}
\newcommand{\Ann}{\mathrm{ann}}
\newcommand{\idm}{\mathbf{I}_n}
\newcommand{\RR}{\mathbb{R}}
\newcommand{\cL}{\mathcal{L}}
\newcommand{\ZZ}{\mathbb{Z}}

\title{Affine cellular algebras and asymptotic algebras}

\author[P. Carvalho]{Paula A.A.B. Carvalho}
\address{Departamento de Matem\'atica, Universidade do Porto, Rua Campo Alegre
  697, 4169-007 Porto, Portugal}
\email{pbcarval@fc.up.pt}

\author[S. Koenig]{Steffen Koenig}
\address{Institute of Algebra and Number Theory, University of Stuttgart,
  Pfaffenwaldring 57, 70569 Stutt\-gart, Germany}
\email{skoenig@mathematik.uni-stuttgart.de}

\author[C. Lomp]{Christian Lomp}
\address{Departamento de Matem\'atica, Universidade do Porto, Rua Campo Alegre
  697, 4169-007 Porto, Portugal}
\email{clomp@fc.up.pt}

\date{\today}

\begin{document}

\begin{abstract}
  The theory of affine cellular algebras $A$ is extended to incorporate
  their asymptotic algebras $\hat{A}$, clarifying unexpected differences
  between classical and affine situations and comparing with Lusztig's
  asymptotic Hecke algebras. The main new results are about a double
  centraliser property between $A$ and $\hat{A}$, about constructing
  $\hat{A}$ from cell modules of $A$, about existence of an embedding
  $A \rightarrow \hat{A}$ and about a faithful functor from torsionless
  $\hat{A}$-modules to $A$-modules as well as about the embedding being weakly
  spectrum preserving (in the sense of Baum and Nistor) and about
  non-zero endomorphisms of cell modules being injective, while there are
  no non-zero homomorphisms between non-isomorphic cell modules.
\end{abstract}

\maketitle

\tableofcontents

\section{Introduction}

The aim of this article is to clarify new and sometimes unexpected
structures of affine cellular algebras and properties of embeddings into
asymptotic algebras, to illustrate these structures and properties by examples
and to
put them into the context of algebraic Lie theory. This continues and builds
upon \cite{KXaffine} and \cite{CKLS}.

Plenty of finite dimensional algebras arising in algebraic Lie theory have
been shown to be cellular algebras. For instance, group algebras of symmetric
groups and their Hecke algebras as well as classical Schur algebras, blocks
of the Bernstein-Gelfand-Gelfand category $\mathcal O$ of semisimple complex
Lie algebras, Temperley Lieb algebras and Brauer algebras are widely studied
examples of cellular algebras, as defined by Graham and Lehrer
\cite{GrahamLehrer}. The definition in \cite{GrahamLehrer}
is using combinatorial properties of a cellular basis that
allows to define and investigate cell modules, which generalise Specht
modules, Weyl modules and Verma modules. A cellular structure allows to
parametrise simple modules in terms of a finite partially ordered set, and to
describe them in terms of linear algebra over a field that has been fixed
from the beginning. 

Subsequently, an equivalent definition of cellularity has been given
\cite{KXcellular} in
terms of a finite chain of  two-sided ideals -- indexed by the same finite
partially ordered set -- with subquotients being direct sums of cell modules.
This led to further ring theoretical and homological properties. The easiest
cellular algebras are split semisimple, more general quasi-hereditary ones
still have finite global dimension and going beyond these, one arrives at
infinite global dimension \cite{KXwhenis}.
In the semisimple or quasi-hereditary case,
subquotients of the cell chain all are idempotent ideals, while in the
non-quasi-hereditary situation some subquotients necessarily have square zero. 

To extend cellularity to not necessarily finite dimensional algebras, in
\cite{KXaffine} affine cellular algebras have been introduced,
and then plenty of algebras have been
shown to be affine cellular, for instance affine Hecke algebras,
KLR algebras, affine Schur algebras and affine Temperley-Lieb algebras.
Here, there is a common affine ground ring and the cell chain still is
finite, indexed by a finite partially ordered set as before. However,
each subquotient comes with its own affine commutative ring and the
subquotient is, as a non-unital algebra, a generalised matrix ring
(introduced by Brown \cite{Brown1955}) over this affine commutative algebra,
with matrix multiplication deformed by adding a fixed matrix as an
intermediate third factor into ordinary matrix multiplication;
this matrix is called swich matrix. Simple modules still are finite
dimensional (now over quotients of the affine rings modulo maximal ideals)
and they are classified by a spectrum, depending on the entries of
the swich matrices. Finite products
of matrix rings over affine commutative rings can
be seen as the analogues of semisimple cellular algebras, and subquotients of
cell ideals that are idempotent produce an analogue of a quasi-hereditary
situation; this motivated Kleshchev to define graded affine quasi-hereditary
algebras \cite{Kleshchev}. In the general situation there are many
other cases than idempotent subquotients or those squaring to zero. As we
will show, the situation is controlled by determinants of swich matrices
(not) being a zero-divisor in the affine commutative
rings of the subquotients. 

Our main focus is on a new ingredient of affine cell theory: Given an affine
cellular algebra A, there exists \cite{KXaffine} an algebra called asymptotic
algebra $\hat{A}$, that is a direct sum of matrix algebras over the
affine rings associated with the subquotients in the cell chain of A, and
by \cite{KXaffine}
there is an algebra homomorphism from A to $\hat{A}$. The subquotients of the
cell chain of A coincide as sets with the components of $\hat{A}$. If A is
finite dimensional cellular in the sense of \cite{GrahamLehrer},
the homomorphism can be
injective only when A is semisimple and coincides with $\hat{A}$, which has
the same dimension. In a different way,
Lusztig had defined in \cite{Lusztig2} asymptotic algebras for Hecke algebras
of finite or affine type over ground rings that are not fields, but for
instance Laurent polynomial rings over the integers. In Lusztig's work and
its applications, asymptotic algebras and embeddings of Hecke algebras into
their asymptotic algebras turned out to be very useful in various respects. 
The general theory of asymptotic algebras of affine cellular algebras has been
started in \cite{KXaffine} and continued in \cite{CKLS}. The current article
aims at extending and deepening this theory, at clarifying the situation in
various applications and at identifying Lusztig's asymptotic $\hat{A}$
with those defined in \cite{KXaffine} in relevant situations.
\medskip

The main results of the theory being developed can be stated informally as
follows: Let $A$ be an affine cellular algebra as defined in \cite{KXaffine},
with a fixed affine cell chain $0=J_0 \subset J_1 \subset J_2 \subset \dots
J_n=A$ and subquotients $J_j/J_{j-1}$ being generalised matrix algebras, where
the multiplication in $J_j/J_{j-1}$ is deformed by adding a fixed middle factor
$\psi^{(j)}$ called swich matrix. Let
$\hat{A}$ be the asymptotic algebra of $A$ and $\iota: A \rightarrow \hat{A}$.
\smallskip

Finite dimensional cellular algebras are either quasi-hereditary
(the desirable case with the best properties), equivalently of finite global
dimension (see \cite{KXwhenis}), or not quasi-hereditary and equivalently
of infinite global dimension. Being quasi-hereditary means that all swich
matrices are invertible, which means their determinants are invertible.
Otherwise, at least some swich matrices have zero determinants.
As we will see, for affine cellular algebras, the desirable case is that the swich matrices are not zero-divisors. 

Under this frequently satisfied assumption, the connection between the
affine cellular algebra $A$ and its asymptotic algebra $\hat{A}$ is
particularly close, as shown by Theorems A, B and C.

\begin{Theox}[special case of Theorem \ref{TheoremAtext}]
  \label{TheoremA}
  Suppose that no determinant of a swich matrix $\psi^{(j)}$
  (for $j = 1, \dots, n$) is a zero-divisor.
  Then there is an embedding $\theta: A \rightarrow \hat{A}$, which is
  explicitly given by the isomorphisms between the endomorphism rings of
  the subquotients ${(J_j/J_{j-1})}_A$ and the generalised matrix rings. 
\end{Theox}

\begin{Theox}[Theorem \ref{TheoremBtext}]
\label{TheoremB}
Suppose that no determinant of a swich matrix $\psi^{(j)}$ (for $j = 1, \dots, n$)
is a zero-divisor and all cell subquotients are idempotent ideals in the
quotient algebras of $A$ modulo lower affine cell ideals. Then the embedding
$\theta: A \rightarrow \hat{A}$ induces a fully
faithful functor from torsionless $\hat{A}$-modules to A-modules.
\end{Theox}

Using $\theta$ yields a K-theoretic consequence. 

\begin{Theox}[Theorem \ref{TheoremCtext}]
\label{TheoremC}
Suppose that no determinant of a swich matrix $\psi^{(j)}$ (for $j = 1, \dots, n$)
is a zero-divisor and all cell subquotients are idempotent ideals in the
quotient algebras of $A$ modulo lower affine cell ideals. Then the
embedding $\theta: A \rightarrow \hat{A}$ is a weakly spectrum preserving
morphism in the sense of Baum and Nistor \cite{BN}.
\end{Theox}

As a consequence, Baum and Nistor's results allow in particular to identify
periodic cyclic homology of the two algebras. 
\medskip

The cell modules of the affine cellular algebra $A$ are direct summands of the
regular $\hat{A}$-module. This is related with another feature of the
asymptotic algebra, and also of smaller analogues of $\hat{A}$. 

\begin{Theox}[Theorem \ref{TheoremDtext}]
\label{TheoremD}
Suppose that no determinant of a swich matrix $\psi^{(j)}$ (for $j = 1, \dots, n$)
is a zero-divisor and all cell subquotients are idempotent ideals in the
quotient algebras of $A$ modulo lower affine cell ideals.

(a) The asymptotic algebra $\hat{A}$ of $A$ can be constructed as
centraliser algebra of endomorphism rings of direct sums of a full set
of cell modules of $A$. In particular, there is
a double centraliser property between $A$ and its asymptotic algebra $\hat{A}$.

(b) The same properties are satisfied for all members of a family of
algebras also including $\hat{A}$ and constructed as centraliser algebras
of endomorphism rings of certain direct sums of cell modules of $A$. 
\end{Theox}

This construction of $\hat{A}$ is quite different from the definition of
asymptotic algebras in \cite{KXaffine}.
\medskip

In the finite dimensional cellular case, homomorphisms between cell modules
rarely are injective, there are no non-zero injective endomorphisms apart
from automorphisms, and if there are no non-zero homomorphisms between
different cell modules at all, the cellular algebra must be semisimple. In
stark contrast to that, for KLR algebras of Dynkin type (shown to be graded
affine quasi-hereditary by Kleshchev, Loubert and Miemietz \cite{KLM}),
Kleshchev and Steinberg \cite{KS} have shown that non-zero endomorphisms
of cell modules exist and always are injective, but
there are no non-zero homomorphisms between different cell modules.
It turns out that from the affine cellular perspective this is not at all
surprising. 

\begin{Theox}[special cases of Theorem \ref{determinante} and Corollary
  \ref{regular}]
\label{TheoremE}
Suppose that subquotients in the affine cell chain are idempotent, their affine commutative rings are integral
domains and the determinants of the swich matrices are non-zero. Then
non-zero endomorphisms of cell modules always are injective and there are
no non-zero homomorphisms between non-isomorphic cell modules.
\end{Theox}

In the final section of this article, we will be
turning to Lie theoretic examples, illustrating the situations occurring in the
new theory and also identifying our asymptotic algebras and the embedding
$\theta$ in various situations, including classical and affine Hecke
algebras studied by Lusztig \cite{Lusztig2}, (see also the book
\cite{GeckJacon}). We also give a worked out small example and
describe in detail the situation
of extended affine Hecke algebras of type $A$, which had been shown
to be affine cellular in \cite{KXaffine}.
Moreover we refer to  various useful methods to verify the assumptions (in
particular on the determinants of swich matrices) of our results above, for
instance a method developed by Mathas in \cite{Mathas} and using the existence
of Jucys-Murphy elements, which can be used to compute Gram determinants
(in our situation aka determinants of swich matrices), as well as $A$ and
$\hat{A}$ becoming isomorphic after extension of scalars (sometimes known as
Lusztig's isomorphism), and using properties of Kazhdan-Lusztig bases.

\section{Determinants for affine Cellular Algebras}
Throughout this text let $k$ denote a commutative Noetherian domain. We recall
the definition of an affine cell ideal of a $k$-algebra, introduced in
\cite{KXaffine}, and then collect basic properties of determinants of
swich matrices. These determinants will play crucial roles in assumptions and
proofs of main results.
\begin{definition} Let $A$ be a unitary $k$-algebra with a $k$-involution $i$ on $A$. A two-sided ideal $J$ in $A$ is called an \emph{affine cell ideal} if and only if $i(J)=J$, such that 
\begin{enumerate}
\item there exist a free $k$-module $V$ of finite rank and an affine commutative $k$-algebra $B$ with identity and with a $k$-involution $\sigma$ such that $\Delta := V \otimes_k B$ is an $A-B$-bimodule, where the right $B$-module structure is induced by that of the right regular $B$-module $B_B$;
\item there is an $A-A$-bimodule isomorphism $\alpha: J \to \Delta \otimes_B \Delta'$, where $\Delta' = B \otimes_k V$ is a $B-A$-
bimodule with the left $B$-structure induced by $_BB$ and with the right $A$-structure via $i$, that is, $(b \otimes v) a := \tau (i(a)(v \otimes b))$ for $a\in A, b\in B$ and $v\in V$, such that 
$\alpha(i(\alpha^{-1}( u \otimes b \otimes v))) = v \otimes \sigma(b) \otimes u$, where one identifies $\Delta\otimes_B\otimes \Delta'$ and $V\otimes_k B \otimes_k V$ and where $\tau$ denotes the flip map.
\end{enumerate}
The module $\Delta$ is called a \emph{cell module} for the affine cell ideal $J$. (In \cite{KXaffine} the notion \emph{cell lattice} is used.) 

The algebra $A$ (with the involution $i$) is called \emph{affine cellular} if and only if there is a $k$-module decomposition 
$A = J'_1 \oplus J'_2 \oplus \cdots \oplus J'_n$ (for some $n$) with $i(J'_j) = J'_j$ for each $j$ and such that setting 
$J_j = \bigoplus_{i=1}^j J'_i$ gives a chain of two-sided ideals of $A$: $0 = J_0 \subset J_1 \subset J_2 \subset \cdots \subset J_n = A$  and for each $1\leq j\leq n$ the quotient $J_j/J_{j-1}$ is an affine cell ideal of $A/J_{j-1}$ (with respect to the involution induced by $i$ on the quotient). This chain is called a \emph{cell chain} for the affine cellular algebra $A$.
\end{definition}

A cellular $k$-algebra is an affine cellular $k$-algebra where all the occurring $k$-algebras $B$ are equal to the given commutative ground ring $k$.
\medskip

Given an affine cell ideal $J$ of $A$, we frequently identify $J$ with $\Delta \otimes_B \Delta'$ via $\alpha$ and the latter with $V\otimes_k B \otimes_k V$. From \cite[Proposition 2.2]{KXaffine} we know that if $J$ is an affine cell ideal of $A$ with involution $i$, then there is a $k$-linear map $\psi : V\otimes_k V \to B$, such that $\sigma (\psi(v, v')) = \psi(v', v)$, and
    \begin{equation}
\alpha\left( \alpha^{-1}(u \otimes b \otimes v) \alpha^{-1}(u' \otimes b' \otimes v') \right) = u \otimes b\psi(v,u') b' \otimes v'        
    \end{equation}
for all $u, u' , v, v' \in V$ and  $b, b' \in B$; and for any element $a\in A$ and $u\otimes b \otimes v \in V\otimes B \otimes V$ we have
\begin{equation}
    a(u \otimes b \otimes v) \in  V \otimes_k  Bb  \otimes_k v, \quad \mbox{ and } \quad (u \otimes b \otimes v)a \in u \otimes_k bB \otimes_k V.
\end{equation}

In particular, fixing a basis $v_1, \ldots, v_n$ for the free $k$-module $V$, we obtain the matrix $\psi = \left( \psi(v_i,v_j)\right)_{i,j} \in M_n(B)$, which is $\sigma$-symmetric in the sense that $\psi = \sigma(\psi)^T$, where $T$ denotes the transposition of matrices and $\sigma$ is extended from $B$ to $M_n(B)$. 
Moreover, $V\otimes_k B\otimes_k V$ gets identified with $M_n(B)$ by sending $v_i\otimes b \otimes v_j$ to $bE_{ij}$, with $E_{ij} \in M_n(B)$ being the elementary matrix with $1$ in the $(i,j)$th entry and $0$ elsewhere. Hence the multiplication of $J$ induces a multiplication on $M_n(B)$ via the $k$-linear isomorphism $\alpha:J\to \Delta\otimes_B \Delta' \cong V\otimes_k B\otimes_k V \cong M_n(B)$, which is given by the matrix $\psi$, namely by
$M\cdot M' = M\psi M'$, for $M,M' \in M_n(B)$. The resulting (not necessarily unital) $k$-algebra will be denoted by $\widetilde{M_n(B)} = (M_n(B),\psi)$ and is called the \emph{generalised matrix ring}. The matrix $\psi$ is also called a \emph{swich} (or \emph{sandwich}) matrix. Thus $\psi$ encodes the multiplication of $J$ (seen as a non-unital ring), but does not have information about the $A$-action on $J$. We will refer to the quintuple $(B,\sigma, \Delta, \psi,\alpha)$ as the \emph{cell datum} of $J$.

Affine cell ideals were introduced to generalise cell ideals of cellular algebras and to provide a framework for studying affine versions of algebras like Hecke or Brauer algebras. A cell ideal of a $k$-algebra $A$ is an affine cell ideal $J$ with cell datum $(k,id, \Delta, \psi, \alpha)$, where $\Delta=V$ is a free $k$-module of rank $n$.  Following Graham and Lehrer \cite{GrahamLehrer}, see also Mathas \cite{Mathas}, a \emph{radical of $\Delta$} is defined as $\rad(\Delta)=\{ x\in \Delta \mid \psi(x,y)=0, \: \forall y\in \Delta\}$. (By definition, this is the
radical of the bilinear form associated with $\psi$. It is in general different
from the radical of the $A$-module ${}_A\Delta$.)

For an affine cell ideal let $\overline{\psi}:\Delta' \otimes_A \Delta \to B$
be given by 
\begin{equation}
    \overline{\psi}(b\otimes v, u\otimes b') := b\psi(v,u)b'
\end{equation}
for $b,b'\in B$ and $v,u\in V$. Then $\overline{\psi}$ is a $B$-bimodule homomorphism that satisfies 
\begin{equation}\label{eq:bilinear-psi}
    \overline{\psi}(y.a, x) = \overline{\psi}(y, a.x),
\end{equation}
for all $y\in \Delta'$, $x\in \Delta$ and $a\in A$.
We define the radical of $\Delta$ to be:
\begin{equation}\label{def:rad}
    \rad(\Delta)=\{ x\in \Delta \mid \overline{\psi}(y, x)=0, \: \forall y\in \Delta'\}.
\end{equation}
The subset $\rad(\Delta)$ is a left $A$-submodule of $\Delta$, since if $x\in \rad(\Delta)$ and $a\in A$, then by (\ref{eq:bilinear-psi}),  
$ \overline{\psi}(y, a.x)=\overline{\psi}(y.a, x)=0$, for any $y\in \Delta'$ and hence $a.x \in \rad(\Delta)$.

For a subset $J\subseteq A$ of a ring $A$ denote by $\ann_Y(J) = \{ y\in Y: Jy=0\}$ the annihilator of $J$ in a left $A$-module $Y$.

\begin{lemma}\label{lemma1}
    Let $\Delta$ be a cell module of an affine cell ideal $J$ with cell datum $(B,\sigma, \Delta, \alpha, \psi)$. Then the following statements are equivalent:
    \begin{enumerate}
        \item[(a)] $\rad(\Delta)=0$;
        \item[(b)] $\psi$ is not a zero-divisor in $M_n(B)$;
        \item[(c)] $\det(\psi)$ is not a zero-divisor in $B$;
        \item[(d)] $\ann_A(J) \cap J = 0$;
        \item[(e)] $\ann_{\Delta}(J)=0$;
        \item[(f)] the ring homomorphism $A\to A/J \times \End(J_A)$ with $a\mapsto (a+J, \lambda_a)$ is injective, where  $\lambda_a$ denotes the left $A$-action of $a$ on $J$.
    \end{enumerate}
\end{lemma}

\begin{proof}
Fix a basis $v_1, \ldots, v_n$ of $V$.

$(a) \Leftrightarrow (b)$: 
Let $x=\sum_{j=1}^n v_j \otimes b_j \in \Delta$. Then for any $1\leq i\leq n$: 
\begin{equation}\label{eq:radical}
 \overline{\psi}(v_i, x) =  \sum_{i=1}^n \psi(v_i, v_j)b_j = \sum_{j=1}^n \psi_{ij}b_j.
\end{equation} 
Hence, $\psi$ is a zero-divisor in $M_n(B)$ if and only if there exists a non-zero matrix $M = \sum_{r,s=1}^n b_{rs} E_{rs}$, such that 
\begin{equation}
\psi M     
= \sum_{i,j,s=1}^n \psi_{ij}b_{js} E_{is}
= \sum_{i,s=1}^n \overline{\psi}\left(v_i, \sum_{j=1}^n v_j\otimes b_{js} \right) E_{is}=0 
\end{equation}
if and only if $\sum_{j=1}^n v_j\otimes b_{js} \in \rad(\Delta)$, for all $1\leq s \leq n$ with at least one $b_{js}$ being non-zero.

$(b)\Leftrightarrow (c)$ follows from \cite[Theorem 9.1]{Brown}; $(c)\Leftrightarrow (d)$ follows from \cite[Lemma 3.2]{CKLS}; 
$(d)\Leftrightarrow (e)$ follows from the fact that $J\cong \Delta^d$, for some $d>0$;
$(d)\Leftrightarrow (f)$ follows from \cite[Proposition 3.14]{CKLS}.
\end{proof}

Let $A$ be an affine cellular $k$-algebra with cell chain
$0 = J_0 \subset J_1 \subset J_2 \subset \cdots \subset J_n = A$. 
Let $J_i/J_{i-1}\cong (M_{d_i}(B_i),\psi^{(i)})$. Recall that in the case of $A$ being a cellular algebra,  $B_i=k$, for all $i$. Denote by $\Delta_i$ the cell module associated to the cell ideal $J_i/J_{i-1}$. In the original definition of cellular algebras, Graham and Lehrer called a $k$-algebra cellular if there exists a \emph{cellular basis}, $\{ a_{st}^i : 1\leq i\leq n, \: 1\leq s,t\leq d_i\}$ satisfying certain axioms (see \cite{GrahamLehrer}). In particular,  $A=\bigoplus_{i=1}^n J'_i$, with $J'_i$ being the free $k$-module with basis $\{a_{st}^i : 1\leq s,t \leq d_i\}$. Moreover, $\Delta_i=V_i$ is the free $k$-module generated by $\{ a_{s1}^i : 1\leq s \leq d_i\}$, which is a left $A$-submodule of $J_i/J_{i-1}$. The bilinear form $\psi^{(i)}:\Delta_i \times \Delta_i \to k$ is determined by the matrix $\left( \psi^{(i)}(a_{s1}^i, a_{t1}^i) \right)_{1\leq s,t\leq d_i}$. Its determinant is called the \emph{Gram determinant} of $\psi^{(i)}$ and denoted by $G(i) = \det\left( \psi^{(i)}(a_{s1}^i, a_{t1}^i) \right)_{1\leq s,t\leq d_i}$. By construction, the Gram determinant coincides with the determinant
of the swich matrix.

\section{Homomorphisms between affine cell ideals.}
The purpose of this section is to show that for affine cellular algebras, whose determinants are not zero-divisors and whose cell ideals are idempotent, there are no non-zero homomorphisms between different cell modules. In particular,
Theorem \ref{TheoremE} will be established. At the end of the section, two
explicit examples of affine cellular algebras will be discussed in detail.

First, some motivation for the results to come will be explained. Let $A$ be  a cellular algebra  over an integral domain $k$ and cell modules $\Delta_i$ and matrices $\psi^{(i)}$, for $1\leq i \leq n$.  Graham and Lehrer showed in \cite[Proposition 2.6]{GrahamLehrer} that $\Hom_A(\Delta_i, \Delta_j)=0$, for $i<j$ and $\psi^{(i)}\neq 0$. However, if $k$ is a field and $\Hom_A(\Delta_i, \Delta_j)
=0$, for all $i \neq j$, the cellular algebra $A$ is semisimple. For affine
cellular algebras we will show that $\Hom_A(\Delta_i, \Delta_j)=0$, for $i<j$
is true under the assumption of the affine cell ideals being idempotent.
Moreover, we will see that
$\Hom_A(\Delta_i, \Delta_j)=0$, for all $i \neq j$, holds true under
assumptions guaranteeing that there is an embedding of $A$ into its
asymptotic algebra $\hat{A}$ (see Theorem \ref{TheoremA}).

Affine quasi-hereditary algebras, defined by Kleshchev \cite{Kleshchev},
are graded algebras satisfying conditions like affine cellular algebras
and in addition further conditions including idempotent generation of the
affine cell ideals. In \cite{KS} the homomorphisms between cell
modules have been investigated for a particular class of affine
quasi-hereditary algebras, KLR-algebras of finite type. Extending work in
\cite{KLM}, Kleshchev and Loubert showed in  \cite{KL} that KLR-algebras $R_{\alpha}$ of finite type are (graded) affine quasi-hereditary over $\ZZ$. For these algebras the affine cell ideals are generated by idempotents (see \cite[Section 4]{KL}).
More precisely, the affine cell chain associated to $R_{\alpha}$ is a chain of ideals $\{I_\pi \mid \pi \in \Pi(\alpha)\}$ indexed by the set of root partitions $\Pi(\alpha)$ of $\alpha$. For each such $\pi \in \Pi$, Kleshchev and Loubert showed that there exist idempotents $\overline{e}_\pi \in \overline{R}_\alpha = R_\alpha/I_{>\pi}$, such that the multiplication         
$$\overline{R}_\alpha \overline{e}_\pi \otimes_{\overline{e}_\pi \overline{R}_\alpha\overline{e}_\pi}  \overline{e}_\pi \overline{R}_\alpha \longrightarrow  \overline{R}_\alpha\overline{e}_\pi \overline{R}_\alpha = \overline{I}_\pi$$
is an isomorphism of $R_{\alpha}$-bimodules. The ideals $\overline{I}_\pi$  are idempotent and finitely generated projective as left (or right) $\overline{R}_\alpha$-modules. Furthermore, the algebras $B_\pi ={\overline{e}_\pi \overline{R}_\alpha\overline{e}_\pi}$ are polynomial rings in finitely many variables over the ground ring, which is supposed to be $\ZZ$ or a field. Hence, the commutative affine algebras $B_\pi$ are integral domains. The cell modules $\Delta_\pi$ associated to the affine cell ideal $I_\pi$ are the left $R_\alpha$-modules $\Delta_\pi = \overline{R}_\alpha \overline{e}_\pi$. (In \cite{KS} the notion \emph{standard module} is used.) 

Kleshchev and Steinberg proved in \cite{KS} that there are no non-zero homomorphisms between different cell (standard) modules $\Delta_\pi$ of a KLR algebra $R_\alpha$ (\cite[Theorem A]{KS}) and that any non-zero endomorphism of $\Delta_\pi$ is injective (\cite[Theorem B]{KS}).  This motivates considering general
(ungraded) affine cellular algebras and to understand when there are no
homomorphisms between the associated cell modules and when non-zero
endomorphisms of cell modules are injective.
\medskip

Let $k$ be a commutative ring, $A$ a $k$-algebra with $k$-linear involution $\imath$ and $J$ an affine cell ideal of $A$. Recall that there exists a finitely generated free $k$-module $V$ of rank $d$ and an affine commutative $k$-algebra  $B$ such that $\Delta=V\otimes_k B$ has a structure of an $A$-$B$-bimodule and there exists an isomorphism of 
$A$-bimodules $\alpha:J\rightarrow \Delta \otimes_B \Delta'$, where $\Delta' = B\otimes_k V$ is a right $A$-module by 
$(b\otimes v)a = \tau\left(\imath(a)(v\otimes b)\right)$.
As a left $A$-module, $J$ is isomorphic to a direct sum of $d$ copies of $\Delta$. Furthermore, the multiplication of $J$ carries over to $\Delta \otimes_B \Delta'$ and the latter can be identified with the generalised 
matrix ring $M_d(B)$ whose multiplication is deformed by a swich matrix $\psi$. More precisely, if $v_1, \ldots, v_d$ is an 
$k$-basis of $V$, then any element of the form 
$v_i\otimes_k b \otimes_B 1\otimes_k v_j$ can be identified with $bE_{ij}$ where $E_{ij}$ denote the elementary matrix 
units. Then $\alpha:J\rightarrow M_d(B)$ satisfies
\begin{equation}\label{eq2} \alpha(xy)=\alpha(x)\psi\alpha(y), \qquad \forall x,y \in J\end{equation}
We identify $\Delta = \Delta \otimes 1 \otimes v_1$ with the first column of $M_n(B)$, i.e. with $\bigoplus_{i=1}^n BE_{i1}$.

For convenience, we include the proof of \cite{CKLS}*{3.3} here. 

\begin{lemma}\label{lemma12} Let $J$ be an affine cell ideal of $A$ with cell datum $(B, \Delta,  d, \psi, \alpha)$. If $J$ is idempotent, then 
$J\Delta=\Delta$ and
\begin{enumerate}
    \item $\Phi:B\to \mathrm{End}({_A\Delta})$ with $b\mapsto \Phi(b):[E_{i1}\to bE_{i1}]$ is an isomorphism of rings.
    \item Every non-zero endomorphism of $\Delta$ is injective if and only if $B$ is an integral domain.
\end{enumerate}
\end{lemma}

\begin{proof}
Let $\alpha: J\rightarrow M_d(B)$ denote the $A$-bimodule isomorphism and identify $\Delta$ with the first column of 
$M_d(B)$, i.e. $\Delta = M_d(B)E_{11} = \bigoplus_{i=1}^d BE_{i1}$.  In order to distinguish the left $A$-module action 
on $\Delta$ from ordinary matrix multiplication, we denote the left $A$-module action on $\Delta$ by $a.x = 
\alpha(a\alpha^{-1}(x))$, for $a\in A$ and $x\in \Delta$. If $a\in J$, then $ a.x = \alpha(a\alpha^{-1}(x)) = \alpha(a) 
\psi x$ by (\ref{eq2}).


Applying $\alpha$ to  $J^2=J$, we get $M_d(B)\psi M_d(B) = M_d(B)$. In particular $$J.\Delta = M_d(B)\psi \Delta = 
M_d(B)\psi M_d(B)E_{11} = M_d(B)E_{11} = \Delta.$$
The map $\Phi: B\rightarrow \mathrm{End}(_A\Delta)$ given by 
$$\Phi(b) : \left[ \sum_{i=1}^n a_iE_{i1}\rightarrow \sum_{i=1}^n a_ibE_{i1}\right],$$ for any $b\in B$, is an 
injective  ring homomorphism. In order to show that $\Phi$ is also surjective consider any left $A$-linear map $f:\Delta 
\rightarrow \Delta$.  Since $J.\Delta=\Delta$, there exist $x_1, \ldots, x_d \in J$ such that  $E_{11} = \sum_{i=1}^d 
x_i . E_{i1} = \sum_{i=1}^d \alpha(x_i) \psi E_{i1} .$
Thus for any $1\leq t\leq d$ we have in $\Delta$: 
$$E_{t1} = E_{t1}E_{11} 
=  \sum_{i=1}^d E_{t1}\alpha(x_i) \psi E_{i1}  
=  \sum_{i=1}^d \alpha^{-1}(E_{t1} \alpha(x_i)) . E_{i1} $$
Applying $f$ and using $A$-linearity yields:
\begin{equation}\label{eq1}f(E_{t1}) 
= \sum_{i=1}^d \alpha^{-1}(E_{t1} \alpha(x_i)) . f\left(E_{i1}\right)
= \sum_{i=1}^d E_{t1} \alpha(x_i) \psi f\left(E_{i1}\right)
= E_{t1}  f\left( \sum_{i=1}^d x_i .E_{i1}\right) = E_{t1} f(E_{11}).
\end{equation}
Suppose $f(E_{11})=\sum_{i=1}^d b_i E_{i1}$ for some $b_i \in B$. Then using equation (\ref{eq1}) for $t=1$ yields:
$$ f(E_{11})=E_{11}f(E_{11})=E_{11} \sum_{i=1}^d b_i E_{i1} = b_1 E_{11}.$$
Hence $f(E_{t1})=E_{t1}b_1 E_{11} = b_1 E_{t1}$, i.e. $f = \Phi(b_1)$, which shows that  $\Phi$ is an 
isomorphism of rings. Hence any endomorphism of $\Delta$ is of the form $\Phi(b)$ for some $b\in B$. In particular, an element $\sum_{i=1}^n a_i E_{i1}$ belongs to the kernel of $\Phi(b)$ if and only if $ba_i=0$ for all $i$. Therefore, $\Phi(b)$ is injective if and only if $b$ is not a zero-divisor. In other words, any non-zero endomorphism is injective if and only if no non-zero element is a zero-divisor, i.e. $B$ is an integral domain.
\end{proof}

The following elementary ring-theoretic Lemma yields information about non-trivial homomorphisms between a left $A$-module $Y$ and an $A/J$-module $X$.

\begin{lemma}\label{lemma2}
Let $A$ be any ring with ideal $J$ and left $A$-modules $X$ and $Y$. Suppose that $JX=0$.
\begin{enumerate}
    \item If $\ann_Y(J)=0$, then $\Hom( {_AX}, {_AY})=0$.
    \item If $JY=Y$, then $\Hom( {_AY}, {_AX})=0$.
    \item If $JY=Y$ and $\ann_Y(J)=0$, then 
    $\End({_AX\oplus Y}) \cong \End({_{A}X}) \times  \End({_AY})$
    as rings. In particular any endomorphism of $X\oplus Y$ is a sum of an endomorphism of $X$ and an endomorphism of 
$Y$, which act componentwise.
\end{enumerate}
\end{lemma}

\begin{proof}
(1) Let $f:X\rightarrow Y$ be a left $A$-linear map. Then, since $JX=0$, $Jf(X)=0$, i.e. $\mathrm{Im}(f) \subseteq \ann_Y(J)=0$ and therefore  $f=0$, i.e.  $\Hom(X,Y)=0$.

(2) Let $f:Y\rightarrow X$ be left $A$-linear, then $f(Y)=f(JY) = Jf(Y)=0$ as $JX=0$. Thus $\Hom(Y,X) = 0$.

(3) As $\End({_AX\oplus Y}) = \End({_AX}) \oplus  \Hom( {_AX}, {_AY}) \oplus  \Hom({_AY},{_AX}) \oplus \End({_AY})$ we have by $(1)$ and $(2)$
$\End({_AX\oplus Y}) = \End({_AX}) \times \End({_AY})$.
\end{proof}

As a consequence of the last two Lemmas we obtain:

\begin{theorem}\label{determinante}
Let $A$ be an affine cellular algebra with cell chain $0=J_0 \subset J_1 \subset \cdots \subset J_n=A$ and cell modules $\Delta_i$ for each $0\leq i\leq n$, i.e. 
$J_i/J_{i-1} \cong \Delta_i \otimes_{B_i}  \Delta'_i \cong \left( M_{d_i}(B_i), \psi^{(i)} \right)$ as $A$-bimodules and 
generalised matrix rings with swich matrix $\psi^{(i)}$.
For any $1\leq i \leq n$ the following statements hold.
\begin{enumerate}
    \item If $J_i/J_{i-1}$ is idempotent, then 
    \begin{enumerate}
        \item $\Hom_A(\Delta_i, \Delta_j)=0$, for all $j>i$.
        \item Every endomorphism of $\Delta_i$ is injective if and only if $B_i$ is an integral domain.
    \end{enumerate}
    \item $\Hom_A(\Delta_j,\Delta_i)=0$, for all $j>i$, if and only if $\det(\psi^{(i)})$ is not a zero-divisor in $B_i$.
\end{enumerate}
\end{theorem}

\begin{proof}
Note that for $j>i$, the module $\Delta_j$ is an $A/J_i$-module, i.e. $J_i\Delta_j = 0$.

(1) If $J_i$ is idempotent, then by Lemma \ref{lemma12}, $J_i\Delta_i=\Delta_i$. By Lemma \ref{lemma2}(2),  $\Hom_A(\Delta_i,\Delta_j)=0$. The second part of (1)  follows from Lemma \ref{lemma12}.

(2) If $\det(\psi^{(i)})$ is not a zero-divisor in $B_i$, then $\Ann_{\Delta_i}(J_i)=0$, by Lemma \ref{lemma1} and by Lemma \ref{lemma2}(1), $\Hom_A(\Delta_j,\Delta_i)=0$ follows for all $j>i$.

Conversely, assume $\Hom_A(\Delta_j,\Delta_i)=0$ for all $j>i$.
Then $\Hom_A(J_j/J_{j-1},\Delta_i)=0$, for all $j>i$. 
If $\det(\psi^{(i)})$ is a zero-divisor, then by Lemma \ref{lemma1}, there is a non-zero $\delta\in \Ann_{\Delta_i}(J_i)$. In particular, $\delta \not\in \Ann_{\Delta_i}(J_n)$, since $1\in J_n=A$. Let $j>i$ be the least number, such that $\delta \not\in \Ann_{\Delta_i}(J_j)$. Then $\delta\in\Ann_{\Delta_i}(J_{j-1})$ by minimality of $j$, and $f\in \Hom_A(J_j/J_{j-1},\Delta_i)$, given by $f(x+J_{j-1})=x\delta$, for $x\in J_j$, is well-defined and non-zero, contradicting our initial assumption. Therefore, $\det(\psi^{(i)})$ is not a zero-divisor in $B_i$.
\end{proof}

Affine Temperley-Lieb algebras are examples of affine cellular algebras, where the affine commutative algebras $B_i$ are integral domains of the form $k[x^{\pm 1}]$ or $k[x]$. In the multiplication of diagrams, there occurs a parameter $q$ which is considered to be an indeterminate over the base ring $k$. The determinant of the appearing associated swich matrix $\psi^{(j)}$ for each of the cells is a polynomial in $q, x_j, x_j^{-1}$ and when considered as a polynomial in $q$ it is
monic, by \cite[Lemma 5.9]{CKLS}. Therefore it is not a zero-divisor. If one wishes to specialise $q$ to a certain value in $k$, then no determinants are zero-divisors for all but finitely many specialisations of $q$.
\medskip

The relation between the non-existence of non-zero homomorphisms between  different cell modules and determinants of the associated swich matrix of each cell being not  a zero-divisor is established in the next Corollary, which follows directly from Theorem \ref{determinante}.

\begin{corollary}\label{regular}
Let $A$ be an affine cellular algebra with cell chain $0=J_0 \subset J_1 \subset \cdots \subset J_n=A$ and cell modules $\Delta_k$ for each $1\leq k\leq n$, i.e. 
$J_k/J_{k-1} \cong \Delta_k \otimes_{B_k}  \Delta_k' \cong \left( M_{d_k}(B_k), \psi^{(k)} \right)$ as $A$-bimodules and 
generalised matrix rings with swich matrix $\psi^{(k)}$.

Then the following holds:
\begin{enumerate}
    \item $\Hom_A(\Delta_i,\Delta_j)=0$, for any $i\neq j$, if $J_i/J_{i-1}$ is idempotent and $\det(\psi^{(i)})$ is not a zero-divisor, for all $i$.
    \item For all $i$, every non-zero endomorphism of $\Delta_i$ is injective,
      provided $J_i/J_{i-1}$ is idempotent  and $B_i$ is an integral domain. 
\end{enumerate}
\end{corollary}

The motivating examples in this context are analogous results in \cite{KS}
for the class of KLR algebras of finite type,
which have been shown to be graded affine quasi-hereditary in \cites{KLM,KL,KS}.
Kleshchev and Steinberg proved in \cite{KS} that for KLR-algebras of
finite type $\Hom_A(\Delta_i,\Delta_j)$ vanishes  for any $i\neq j$ and that
for all $i$, every non-zero endomorphism of $\Delta_i$ is injective. Here,
$\Delta_i$ refers to graded standard modules in the graded quasi-hereditary
and hence also graded cellular structure. The methods used in \cite{KS} are
different from those used to obtain  Corollary \ref{regular} above. 

The graded standard modules in \cite{KS} are standard modules in the sense
of \cite{KXaffine}, hence (affine) cell modules. 
Since all affine commutative algebras $B_i$ of the affine cell chain
of a KLR algebra found in \cite{KL} are polynomial algebras and hence integral
domains, it follows from Theorem \ref{determinante}, that any non-zero
endomorphism of a cell ideal $\Delta_i$ is injective. This implies
\cite[Theorem B]{KS}.
\medskip

The following two explicit examples illustrate the situation and also provide
affine cellular algebras occurring in applications. 

\begin{example}
  The integral Schur algebra $A:=S_{\mathbb Z}(2,2)$ with $k=\mathbb{Z}$ can
  be represented as 
  $$A=\left\{  \left(\begin{array}{ccc} a & b & 0\\ 2c & d & 0 \\ 0 & 0 & e\end{array}\right) \in M_3(\ZZ)  \mid  c\in \ZZ,\:\: d = e (\mathrm{mod}\: 2)\right\}.$$ 
  (See \cite{Green} for details on Schur algebras, including multiplication
  formulae, structure and their role in representation theory of the
  algebraic groups $GL(n)$ and the symmetric groups $\Sigma_r$.)

Then $A$ is a $\ZZ$-algebra with involution $i$ given by $\left(\begin{array}{ccc} a & b & 0 \\ 2c & d & 0 \\ 0&0&e\end{array}\right)  \mapsto \left(\begin{array}{ccc} a & c & 0 \\ 2b & d & 0 \\ 0 & 0 & e\end{array}\right)$. The ideal 
    $J_1=  \left(\begin{array}{ccc} \ZZ & \ZZ & 0 \\ 2\ZZ & 2\ZZ &0 \\ 0 & 0 & 0 \end{array}\right)$ 
    satisfies $A/J_1\cong \ZZ$. Set $J_2=A$.
    Then $J_1$ is an affine cell ideal of $A$ with $B=\ZZ$, $\Delta_1=B^2$ and $A$-isomorphism $\alpha: J_1 \rightarrow \Delta_1\otimes_B (\Delta_1)' \cong (M_2(\ZZ),\psi^{(1)})$, given by $\left(\begin{array}{ccc} a & b & 0\\ 2c & 2d & 0 \\ 0 & 0 & 0\end{array}\right) \mapsto  \left(\begin{array}{cc} a & b \\ c & d \end{array}\right)$, with 
    $\psi^{(1)} = \left(\begin{array}{cc} 1 & 0 \\ 0 & 2\end{array}\right)$. The determinant of $\psi^{(1)}$ equals $2$, which is not a zero-divisor and $J_1$ is idempotent. The swich
    matrix $\psi^{(2)}$ is the identity matrix $(1)$.  

    The two cell modules $\Delta_1 \cong B^2$ and $\Delta_2 \cong B$ both have
    endomorphism ring $\ZZ$ and there are no non-zero homomorphisms between
    them. 
\end{example}
\begin{example}\label{example}
Let $k$ be a field and $A$ the $k$-algebra 
$$A  = \left\{\left(\begin{array}{ccc} x & a & b\\ 0 & c & d \\ 0 & 0 & x\end{array}\right) \mid x,a,b,c,d\in k\right\} \supset 
\left(\begin{array}{ccc} 0 & k & k\\ 0 & k & k \\ 0 & 0 & 0\end{array}\right)=:J_1.$$
When $k$ is an infinite field of characteristic two, $A$ is the Schur algebra
$S_k(2,2)$. Over a field of complex numbers, $A$ is a regular block, for
instance the principal block, of the Bernstein-Gelfand-Gelfand category
$\mathcal{O}$ of the semisimple complex Lie algebra
${\mathfrak{sl}}(2,{\mathbb{C}})$. Over any field, $A$ is the Auslander algebra
of the self-injective algebra $B=k[x]/(x^2)$, that is,
$A = \End_B(B \oplus k)$ with $B$ and $k$ representing the isomorphism
classes of indecomposable $B$-modules.

Let $i$ be the involution of $A$ defined by $\left(\begin{array}{ccc} x & a & b\\ 0 & c & d \\ 0 & 0 & x\end{array}\right) \mapsto \left(\begin{array}{ccc} x & d & b\\ 0 & c & a \\ 0 & 0 & x\end{array}\right)$. Then $i(J_1)=J_1$ holds and, as a two-sided ideal, $J_1$ is generated by the idempotent $e=E_{22}$, i.e. $J_1=A e A$. Let $\Delta_1 = A e$ and ${\Delta'}_1 = i(\Delta_1) = e A$. Thus the multiplication in $A$ yields an isomorphism $\Delta_1 \otimes_k {\Delta'}_1 \cong A e A = J_1$, by 
$$\left(\begin{array}{ccc} 0 & r & 0\\ 0 & s & 0 \\ 0 & 0 & 0\end{array}\right) \otimes \left(\begin{array}{ccc} 0 & 0 & 0\\ 0 & t & v \\ 0 & 0 & 0\end{array}\right) \mapsto  
\left(\begin{array}{ccc} 0 & rt & rv\\ 0 & st & sv \\ 0 & 0 & 0\end{array}\right).
$$
 Let $v_1, v_2$ be a basis of $\Delta_1$, where $v_1=e=E_{22}$  and $v_2 = E_{12}$. Then the basis of ${\Delta'}_1 = i(\Delta_1)$ is $w_1=i(v_1)=v_1$ and $w_2=i(v_2)=E_{23}$. Consider an arbitrary element $av_1\otimes w_1 + b v_1\otimes w_2 + cv_2\otimes w_1 + dv_2\otimes w_2$ of $\Delta_1 \otimes_k {\Delta'}_1$ as a matrix $\left( \begin{array}{cc}
a & b \\ c & d \end{array}\right)$. Then the inverse of the isomorphism between  $\Delta_1 \otimes_k {\Delta'}_1$ and $J_1$ is given by
$\alpha:J_1 \to \Delta_1 \otimes_k (\Delta_1)'$ sending
$$
\left( \begin{array}{ccc}
0 & a & b \\ 0 & c & d \\ 0 & 0 & 0
\end{array}\right)
\mapsto \left( \begin{array}{cc}
c & d \\ a & b 
\end{array}\right) $$
Moreover, $\alpha$ satisfies the compatibility condition with the involution:
$$\alpha \left( 
i\left( \begin{array}{ccc}
0 & a & b \\ 0 & c & d \\ 0 & 0 & 0
\end{array}\right)\right)
 = \left( 
\begin{array}{cc}
c & a \\  d & b 
\end{array}\right)
= \left( 
\begin{array}{cc}
c & d \\  a & b 
\end{array}\right)^T
= \left(\alpha \left( \begin{array}{ccc}
0 & a & b \\ 0 & c & d \\ 0 & 0 & 0
\end{array}\right)\right)^T
$$
This shows that $J_1$ is an affine cell ideal of $A$. The quotient $\Delta_2:=A/J_1\cong k$ is a $1$-dimensional affine cell ideal. Note that $\Delta_2 \to \Delta_1$ defined by $1+J_1 \mapsto E_{12}$ is a non-zero homomorphism, i.e. $\mathrm{Hom}_{A}(\Delta_2, \Delta_1)\neq 0$.
The  matrix $\psi^{(1)}$ can be deduced from the multiplication of $J$, since
$$\left(
\begin{array}{cc}
a & b \\  c & d
\end{array}
\right)
\psi 
\left(
\begin{array}{cc}
r & s \\  t & u
\end{array}
\right) = 
\alpha\left(
\alpha^{-1}\left(
\begin{array}{cc}
a & b \\  c & d
\end{array}
\right)
\alpha^{-1}\left(
\begin{array}{cc}
r & s \\  t & u
\end{array}
\right) \right)
=\left(
 \begin{array}{cc}
ar & as \\  cr & cs
\end{array}
\right),$$
for any $2\times 2$-matrices. Hence $\psi^ {(1)}=\left(
\begin{array}{cc}
1 & 0 \\  0 & 0
\end{array}
\right)$, which has determinant zero.
\end{example}

\section{The asymptotic algebra of an affine cellular algebra}

In this section, the asymptotic algebra of an affine cellular algebra
and the embedding of an affine cellular algebra into its asymptotic algebra
will be described in detail. Moreover, the asymptotic algebra will be
constructed as endomorphism ring in a double centraliser property.
This proves Theorems A and D. 

Let $A$ be an affine cellular algebra with cell chain 
\begin{equation}\label{notation:cellchain}
  0=J_0 \subset J_1 \subset \cdots \subset J_n=A,\end{equation} with cell data 
$(B_i, \sigma_i, \Delta_i, \psi^{(i)}, \alpha_i)$. 
In particular, 
\begin{equation}\label{notation:celldata}\alpha_i: J_i/J_{i-1} \to \Delta_i\otimes_ {B_i} \Delta_i' = (M_{d_i}(B_i),\psi^{(i)})\end{equation}
is an isomorphism of $A$-bimodules.

The \emph{asymptotic algebra} $\asym$ of an affine cellular $A$ has been
defined in \cite{KXaffine} as the direct product of all matrix algebras
$M_{d_i}(B_i)$ obtained from the cell data of the cell chain. In
\cite[comment before Lemma 5.5 on page 177]{KXaffine}, it has been mentioned
that the definition of the asymptotic algebra given in \cite{KXaffine}
coincides with Lusztig's asymptotic algebra in the case of the extended affine
Hecke algebra of type $A_{n-1}$. Further details will be given in the final
section when discussing examples and methods.


\begin{definition}
  The asymptotic algebra of $A$ is defined as 
$\asym = M_{d_n}(B_n) \times \cdots \times M_{d_1}(B_1)$. Set $\asym_i = M_{d_i}(B_i)$, for any $i$.
\end{definition}

For each $1\leq k \leq n$ denote the left multiplication of $A$ on $J_k/J_{k-1}$
by 
$\lambda_k: A \to \End( {J_k/J_{k-1}} )$, i.e. 
$$\lambda_k(a)=[x+J_{k-1} \mapsto ax + J_{k-1}],$$ for all $a\in A$ and $x\in J_k$. Note that $\lambda_k$ is right $A$-linear. Consider the ring homomorphism
\begin{equation}
    \lambda: A \to \End( {J_n/J_{n-1}} ) \times \End( {J_{n-1}/J_{n-2}} ) \times \cdots \times \End( {J_1} ), \qquad \lambda = (\lambda_n, \lambda_{n-1}, \cdots, \lambda_1)
\end{equation}

Due to \cite[Lemmas 4.2, 3.2]{CKLS}, the map $\lambda$ is injective if and only if no determinant $\det(\psi^{(k)})$ is a zero-divisor in $B_k$. By \cite[Theorem 3.3]{CKLS}, the ring $\End( {J_k/J_{k-1}})$ is isomorphic to $M_{d_k}(B_k)$ if $\det(\psi^{(k)})$ is not a zero-divisor in $B_k$ (or if  $J_k/J_{k-1}$ is idempotent).  We will briefly recall the steps of the proof of \cite[Theorem 3.3]{CKLS}. Let $J$ be an affine cell ideal with cell datum $(B,\sigma, \Delta, \psi,\alpha)$ with $n=\rank_B(\Delta)$. Assume that $J$ is idempotent or $\det(\psi)$ is not a zero-divisor in $B$.  Then in a first step one shows that the right $A$-linear endomorphisms of $J$ coincide with the right $J$-linear endomorphisms  of $J$, i.e. $\End(J_A)=\End(J_J)$. Moreover, the isomorphism $\alpha: J \to (M_n(B),\psi)=:\widetilde{M_n(B)}$ induces an isomorphism between $\End(J_J)$ and $\End({\widetilde{M_n(B)}_{\widetilde{M_n(B)}}})$ given by $f\mapsto \alpha f \alpha^{-1}$.  Finally one checks that the ring homomorphism $l:M_n(B)\to \End({\widetilde{M_n(B)}_{\widetilde{M_n(B)}}})$ given by left multiplication, $l(X)=[Y\mapsto XY]$, for any matrices $X,Y \in M_n(B)$ is invertible, with inverse given by $l^{-1}(g) = g(\idm)$, for any $g\in \End({\widetilde{M_n(B)}_{\widetilde{M_n(B)}}})$. Hence the isomorphism between $\End(J_A)$ and $M_n(B)$ is given by $f\mapsto \alpha(f(\alpha^{-1}(\idm)))$.

We summarise the situation as follows:
The isomorphism $$\alpha: J \to (M_n(B),\psi)= \widetilde{M_n(B)}$$
induces an isomorphism $\End(J_A) \cong  M_n(B), 
f  \mapsto  \alpha(f(\alpha^{-1}(\idm)))$,                       
                        which is the composition \medskip

\begin{center}
$\begin{array}{ccccc}
  \End(J_A) = \End(J_J) & \stackrel{\alpha}{\rightarrow} &
\End({\widetilde{M_n(B)}_{\widetilde{M_n(B)}}}) &
\stackrel{l^{-1}}{\rightarrow} & M_n(B) \\
f & \mapsto & \alpha f \alpha^{-1} & & \\
& & g & \mapsto & g(\idm)
 \end{array}$
 \end{center}

\medskip

\noindent For $1\leq k \leq n$ let $e_k = \alpha_k^{-1}(\mathbf{I}_{d_k})\in J_k/J_{k-1}$. Combining the left multiplication ${\lambda_k:A\to \End(J_k/J_{k-1})}$ with the isomorphism between $\End(J_k/J_{k-1})$ and $M_{d_k}(B_k)$ yields a ring homomorphism
$\Theta_k: A \to M_{d_k}(B_k)$ given by 
\begin{equation}\Theta_k(a)=\alpha_k(ae_k),\end{equation}
for all $a\in A$. Since the multiplication on $J_k/J_{k-1}$ is encoded in $\psi^{(k)}$, by (\ref{eq2}) any element $a\in J'_k$ satisfies:
\begin{equation}\label{eq:image-theta-k}\Theta_k(a)=\alpha_k(ae_k) = \alpha_k(\overline{a})\psi^{(k)},\end{equation}
where $\overline{a} = a + J_{k-1}$.

\begin{theorem}[Theorem \ref{TheoremA}] \label{TheoremAtext}
    Let $A$ be an affine cellular algebra such that no determinant $\det(\psi^{(k)})$ is a zero-divisor in $B_k$. Then $\Theta:A\to \asym$ given by $\Theta = (\Theta_n, \cdots, \Theta_1)$, is an embedding of $A$ into its asymptotic algebra. 
    More concretely,
\begin{equation}\label{notation:embedding}
    \Theta(a) = (\alpha_n(ae_n), \ldots, \alpha_1(ae_1)), \qquad a\in A.
\end{equation}
Let $1\leq k \leq n$, then for any $a\in J'_k$:
\begin{equation}
    \Theta(a) = (\alpha_n(ae_n), \cdots, \alpha_{k}(ae_{k}), \alpha_k(\overline{a})\psi^{(k)}, 0, \ldots, 0), 
\end{equation}
since $\Theta_l(a)=0$, for all $l<k$.
\end{theorem}

\begin{remark} 
  The same proof shows that there is the following embedding
  of $A$ into a more economic version of $\asym$: 
  
  Let $m$ be the least non-negative integer such that $J_m/J_{m-1}$ is a
  faithful left $A/J_{m-1}$-module. Then there is an embedding of $A$
  into $M_{d_m}(B_m) \times \cdots \times M_{d_1}(B_1)$.
  
 In particular, when the affine cell ideal $J_1$ is a faithful $A$-module, one
 may choose $m=1$ and get an embedding into the matrix algebra
 $M_{d_1}(B_1)$.

 We choose however not to take this abbreviation, for the following
 reason: By \cite{KXaffine}, simple $A$-modules are classified by
 subsets of the disjoint union of the maximal spectra of the matrix
 components of $\asym$. The abbreviation using $m < n$ would confuse this
 approach and Theorem \ref{TheoremC} could not be formulated any more.
\end{remark}



In the second part of this section we will show that the asymptotic algebra can also be constructed as the double centraliser of the direct sum of cell modules. 
Let $J$ be an affine cell ideal of $A$, with cell datum $(B,\sigma,\Delta, \psi,\alpha)$ and $n=\rank_B(\Delta)$. Let ${\alpha:J \to \Delta \otimes_B \Delta'}$ be the given $A$-bimodule isomorphism with $\Delta = V\otimes B$ an $A$-$B$-bimodule, $V$ free of rank $n$ over $k$ and $\Delta' = B\otimes V$ the induced $B$-$A$-bimodule. Then $J\cong \Delta^n$ as left $A$-module and $J \cong (\Delta')^n$ as right $A$-module. Moreover, ${_B\Delta} \cong  {_BB}^n \cong {\Delta'_B}$ as $B$-modules.

By Lemma \ref{lemma12} or \cite[Theorem 3.3]{CKLS},  $\End({_A\Delta}) \cong B$, if $J$ is idempotent or $\det(\psi)$ is not a zero-divisor in $B$. 
Similarly,  the map $\mathbf{l}_b:\Delta' \to \Delta'$ given by $\mathbf{l}_b(x)=bx$, for any $b\in B$ is  right $A$-linear and yields an injective ring homomorphism $\mathbf{l}:B\to \End({\Delta'_A})$. 
In case $J$ is idempotent or $\det(\psi)$ is not a zero-divisor in $B$, it can be shown as in Lemma \ref{lemma12} that the map $\mathbf{l}$ is an isomorphism.
Thus in this case we obtain the following chain of isomorphisms, using $J\cong (\Delta')^n$ as right $A$-modules:
\begin{equation}\label{eqn1}\End(J_A) \cong M_n(\End(\Delta'_A)) \cong  M_n(B) \cong \End(B^n_B) = \End(\Delta_B).\end{equation}

The left multiplication of $A$ on $J$, i.e. $\lambda(a)=:\lambda_a\in \End(J_A) \cong \End(\Delta_B)$ is defined by $\lambda_a(x)=ax$, for all $x\in J$. 
Identifying $J\cong  \Delta \otimes_B \Delta' \cong \Delta^n$, the left $A$-action $\lambda_a$ on $J$ and its restriction, say $\lambda'_a$, on $\Delta$ coincide. 
The map $\Theta:A\to A/J\times \End(J_A)$, with $\Theta(a) = (a+J, \lambda_a)$, for all $a\in A$, is a ring homomorphism. By Lemma \ref{lemma1}, $\Theta$ is injective if and only if $\det(\psi)$ is not a zero-divisor in $B$. Using equation (\ref{eqn1}) we obtain an injective ring homomorphism
$$\Theta:A \to A/J \times \End(\Delta_B),\qquad  a \mapsto (a+J, \lambda'_a),$$
with $\lambda'_a$ the left action of $A$ on $\Delta$. Repeating this process along a cell chain yields an embedding of rings:
$$\Theta:A \to \prod_{i=1}^n \End( (\Delta_i)_{B_i}), \qquad a \mapsto (\lambda'^{(n)}_a, \cdots , \lambda'^{(1)}_a),$$
where $\lambda'^{(i)}_a$ denotes the left action of $a\in A$ on $\Delta_i$.
By Corollary \ref{regular}, $\Hom_A(\Delta_i,\Delta_j)=0$ for $i\neq j$, if for all $i$ the ideal $J_i/J_{i-1}$ is idempotent and $\det(\psi^{(i)})$ is not a zero-divisor in $B_i$. Therefore, the faithful $A$-module $M=\bigoplus_{i=1}^n \Delta_i$ satisfies:
\begin{eqnarray}
\End_A(M)&=&\prod_{i=1}^n \End_A(\Delta_i)\\
 \End(M_{\End_A(M)}) &=& \prod_{i=1}^n \End( {\Delta_i}_{B_i})
\end{eqnarray}
Thus we have proven the following Theorem:

\begin{theorem}\label{TheoremDtext}
  Let $A$ be an affine cellular algebra with cell chain $0=J_0 \subset J_1 \subset \cdots \subset J_n=A$ and cell datum $(B_i,\sigma_i,\Delta_i, \psi^{(i)}, \alpha_i)$ and $n_i=\rank_{B_i}(\Delta_i)$, for all $i$.
If for all $i$, the ideal $J_i/J_{i-1}$ is idempotent and $\det(\psi^{(i)})$ is not a zero-divisor, then the following hold.
\begin{enumerate}
    \item $M=\bigoplus_{i=1}^n \Delta_i$ is a faithful left $A$-module.
    \item $T=\End_A(M)=\prod_{i=1}^n \End_A(\Delta_i)$.
    \item The left $A$-module action of $A$ on $M$ induces an embedding of rings
    $$\lambda: A \to \End(M_T) = \prod_{i=1}^n \End( {\Delta_i}_{B_i}) \cong M_{d_n}(B_n)\times \cdots \times M_{d_1}(B_1) = \asym,$$
    which coincides with the embedding of $A\to \asym$ as shown in \cite{CKLS}.
\end{enumerate} 
\end{theorem}
\begin{proof}
Note that $M$ is a faithful left $A$-module, because if $aM=0$, then $a\Delta_i=0$, for all $i$. Hence, $a(J_i/J_{i-1})=0$, for all $i$. Thus $a + J_{i-1}\in \ann_{J_{i-1}/J_{i-2}}(J_{i-1}/J_{i-2})$. As $\mathrm{det}(\psi^{(i-1)})$ are not
zero-divisors by assumption,  $a\in J_{i-1}$, for all $i$. Thus $a\in J_{0}=0$. In particular, the left action of $A$ on $M$ yields an embedding of $A$ into $\End(M_T)$, where $T=\End(_AM)$. By the previous argument, $\End(M_T)\cong \prod_{i=1}^n \End({\Delta_i}_{B_i}) \cong \asym$.
\end{proof}


\section{Properties of the embedding into the asymptotic algebra}

In this section we will provide further information related with the embedding
into the asymptotic algebra by considering the restriction functor associated to the
embedding $A\to \asym$. Moreover we will derive a geometric property of
the embedding. This enhances in particular Theorems \ref{TheoremB} and
\ref{TheoremC}.

Throughout this section, $A$ is an affine cellular algebra with cell chain $\{J_i : 1\leq i \leq n\}$ as in  (\ref{notation:cellchain}) and cell data $(B_i,\sigma_i,\Delta_i, \psi^{(i)}, \alpha_i)$ as in (\ref{notation:celldata}), with $d_i=\rank_{B_i}(\Delta_i)$, such that no determinant $\det(\psi^{(i)})$ is a zero-divisor in $B_i$ and such that the embedding $\Theta: A \to \asym$ into its asymptotic algebra $\asym=\asym_n \times \cdots \times \asym_1$, with $\asym_k=M_{d_k}(B_k)$, is given by (\ref{notation:embedding}), i.e. $\Theta(a)=(\Theta_n(a),\ldots, \Theta_1(a))$, for all $a\in A$. Then, for $a\in J_k'$, we have $\Theta_k =  \alpha_k(\overline{a})\psi^{(k)}$, where $\overline{a}=a+J_{k-1}$. The embedding $\Theta$ induces a restriction functor $$F: \asym\mathrm{-Mod} \to A\mathrm{-Mod},$$ where for $X\in \asym\mathrm{-Mod}$ the additive group $X$ becomes a left $A$-module $F(X)$ by letting
\begin{equation}\label{eq:def-module-structure}
 a\cdot x := \Theta(a)x, \qquad \forall a\in A, \: x\in X.
\end{equation}
All components $\asym_k = M_{d_k}(B_k)$ are unital algebras, with identity $\mathbf{I}_{d_k}$. The identity element of $\asym$ is $\left(\mathbf{I}_{d_n},\ldots, \mathbf{I}_{d_1}\right)$. We will identify elements $a\in \asym_k$ with the element $a\mathbf{I}_{d_k} = (0,\cdots, a, \cdots, 0) \in \asym$ and any left $\asym$-module $X$ decomposes into $X=X_n\times \cdots \times X_1$, where $X_k = \asym \mathbf{I}_{d_k} X =  \asym_k X$ is a left $\asym_k$-module. Hence, any element of $X$ can be uniquely written as $x=(x_n,\ldots, x_1)$ with $x_i\in X_i$ and the $A$-action of $a\in A$ on $x\in F(X)$ is given by $a\cdot x = \Theta(a)x = (\Theta_n(a) x_n, \ldots, \Theta_1(a)x_1)$. Moreover, for $\asym$-modules $X$ and $Y$, 
$\Hom_{\asym}(X,Y) = \prod_{k=1}^n \Hom_{\asym_k}(X_k,Y_k).$ 
Recall that by definition of an affine cellular algebra $A=\bigoplus_{i=1}^n J'_i$ as $k$-module, where $J_i = \bigoplus_{j\leq i} J'_j$.

\begin{lemma}\label{lem:idempotent} 
Let $1\leq k ,l \leq n$ and assume that $J_k/J_{k-1}$ is idempotent. The following statements hold for an $\asym_k$-module $X$ and an $\asym_l$-module $Y$.
\begin{enumerate}
    \item $J'_k\cdot X = X$.
    \item $\Hom_A(F(X),F(Y))=0$, if $k<l$.
    \item $\Hom_A(F(X),F(Y))=\Hom_{\asym}(X,Y)$, if $k=l$.
\end{enumerate}
\end{lemma}

\begin{proof}
Note, that the $A$-action on $X$ is completely determined by $\Theta_k$, i.e. $a\cdot x = \Theta(a)x = \Theta_k(a)x$, for any $a\in A$ and $x\in X$, because $\asym_iX=0$, for $i\neq k$.
Moreover, since $J_k/J_{k-1}$ is idempotent, $\asym_k\psi^{(k)}\asym_k=\asym_k$ holds.

(1) Using (\ref{eq:image-theta-k}), we have  $J_k' \cdot X= \Theta_k(J_k')\asym_k X =\asym_k\psi^{(k)}\asym_kX = \asym_k X = X.$

(2) Let $f:F(X)\to F(Y)$ be $A$-linear. By (1) and $\Theta_l(J'_k)=0$, for $l>k$:
$$f(X)=f(J'_k\cdot X)=J'_k\cdot f(X) \subseteq J'_k\cdot Y = \Theta_l(J'_k)Y=0.$$

(3) We always have $\Hom_{\asym}(X,Y)\subseteq \Hom_{A}(F(X),F(Y))$. Assume $k=l$ and let $f:F(X)\to F(Y)$ be  $A$-linear. We need to show that $f$ is $\asym$-linear. 
Since $\alpha_k(J_k/J_{k-1})=\asym_k$, for any $y \in \asym_k$, there exists $\tilde{y}\in J_k'$ such that
$\Theta_k(\tilde{y}) = \alpha_k(\overline{\tilde{y}})\psi^{(k)}=y\psi^{(k)}.$ In particular, $a(\tilde{y}\cdot z) = ay\psi^{(k)}z = \widetilde{ay}\cdot z$, for any $a,y,z\in \asym_k$.
Furthermore, since $\asym_k\psi^{(k)}\asym_k = \asym_k$, there exist elements $y_i, z_i\in \asym_k$, such that 
\begin{equation} \mathbf{I}_{d_k} 
    = \sum_{i=1}^t y_i \psi^{(k)} z_i
        =  \sum_{i=1}^t \tilde{y_i}\cdot z_i.\end{equation}

Thus, any element $a\in \asym_k$ can be written as $a=\sum_{i=1}^t \widetilde{ay_i}\cdot z_i$. Then, for $x\in X$:
\begin{eqnarray*}
    f(ax)
    = f\left(\sum_{i=1}^t \widetilde{ay_i} \cdot z_ix\right)
    = \sum_{i=1}^t \widetilde{ay_i} \cdot f ( z_ix )
    = \sum_{i=1}^t a \tilde{y_i}\cdot f ( z_ix )
    = a f\left(\sum_{i=1}^t \tilde{y_i}\cdot z_ix \right) = af(x).
\end{eqnarray*}
Hence $f$ is $\asym_k$-linear and also $\asym$-linear, since $\asym_i X = 0$, for $i\neq k$.
\end{proof}

A module is said to be {\em torsionless} provided it can be embedded into a projective (free) module. 

\begin{lemma}\label{lem:torsionless} Let  $1\leq l < k \leq m$,  $X$ a  $\asym_k$-module and $Y$ a torsionless $\asym_l$-module.  If $\det(\psi^{(l)})$ is a not a zero-divisor in $B_l$, then  $\Hom_A(F(X),F(Y))=0$. \end{lemma}

\begin{proof}
Since $\det(\psi^{(l)})$ is a not a zero-divisor in $B_l$, $\Ann_{P}(\psi^{(l)})=0$, for any free $\asym_l$-module $P=\asym_l^{(I)}$. In particular, 
$\Ann_{Y}(\psi^{(l)})=0$, for any torsionless $\asym_l$-module $Y$.

Let $f\in \Hom_A(F(X),F(Y))$. Then $\psi^{(l)} f(X) \subseteq  J_l' \cdot f(X) = f(J'_l\cdot X) = f(\Theta_k(J_l')X) = 0$, since  $\Theta_k(J_l')=0$ as $l<k$.  Thus, $f(X)\subseteq \Ann_Y(\psi^{(l)})=0$ and $f=0$.
\end{proof}

Let $\cT(\asym)$ denote the class of torsionless left $\asym$-modules. The restriction
functor $F:\cT(\asym) \to A\mathrm{-Mod}$ is fully faithful on this subcategory.

\begin{theorem}[Theorem \ref{TheoremB}] \label{TheoremBtext}
  Let $A$ be an affine cellular algebra with an affine cell chain such that
  all occurring affine cell ideals are idempotent and no determinant of their
  respective swich matrix is a zero-divisor.
  Then the embedding
  $\Theta: A \to \asym$ induces a fully faithful restriction functor
  $F: \cT(\asym) \to A\mathrm{-Mod}$, between torsionless $\asym$-modules and
  $A$-modules.
\end{theorem} 
\begin{proof}
  Let $X=\prod_{i=1}^n X_i$ and $Y=\prod_{i=1}^n Y_i$ be any torsionless
  $\asym$-modules. 
  As  $J_k/J_{k-1}$ is idempotent, $\Hom_A(F(X_k),F(Y_l)) = 0$ for
  all $k<l$, by Lemma \ref{lem:idempotent}.
By Lemma \ref{lem:torsionless}, $\Hom_A(F(X_k),F(Y_l)) = 0$ for $k>l$. Furthermore, $\Hom_A(F(X_k),F(Y_k) = \Hom_{\asym}(X_k,Y_k)$, for all $k$, by Lemma \ref{lem:idempotent}. Therefore,
$$\Hom_A(F(X),F(Y)) = \bigoplus_{k,l=1}^n \Hom_A(F(X_k),F(Y_l)) = \bigoplus_{k=1}^n \Hom_{\asym}(X_k,Y_k) = \Hom_{\asym}(X,Y).$$
\end{proof}


In the second part of this section we will compare the primitive spectrum of an affine cellular algebra with the primitive spectrum of its asymptotic algebra.
In \cite{BN} Baum and Nistor determined the periodic cyclic homology of the Iwahori-Hecke algebra $H_q$ at a proper root of unity $q$.  Their method is based on a general result, which states that a map called \emph{weakly spectrum preserving morphism} of finite type algebras induces an isomorphism in periodic cyclic homology. Furthermore, they showed that Lusztig's embedding $\phi_q: H_q \to J$, of the Iwahori–Hecke algebra $H_q$ into Lusztig's asymptotic algebra $J$ is weakly spectrum preserving. Here we will show that the embedding of an affine cellular algebra into its asymptotic algebra is weakly spectrum preserving if and only if all affine cell ideals appearing in the affine cell chain of the algebra are idempotent ideals (not necessarily generated by idempotent elements).

Let $R$ be a ring and denote by $\Prim(R)$ the set of (left) primitive ideals of $R$, where an ideal $P$ is called left primitive if it is the annihilator of a simple left $R$-module. Over a not necessarily unital ring $R$, a module $V$ is said to be simple if $R V\neq 0$ and $R v=V$ for all  $v\in V\backslash\{0\}$.

\begin{definition}[\cite{BN}]
  A $k$-algebra homomorphism  $\Theta:R\to S$ between two (not necessarily unital) $k$-algebras $R$ and $S$ is called \emph{spectrum preserving} if for any
  $P\in \Prim(S)$, $\Theta^{-1}(P)$ is contained in a unique primitive ideal of $R$ and the resulting map $\Theta^*: \Prim(S)\to \Prim(R)$ is a bijection.
\end{definition}

Given $B$ an affine commutative $k$-algebra, $n\geq 1$ and $\psi\in M_n(B)$, we denote, as above, the swich algebra $(M_n(B),\psi)$ by $\widetilde{M_n(B)}$. The map
$$ \varphi: \widetilde{M_n(B)} \to M_n(B), \qquad x\mapsto \varphi(x)=x\psi$$ is a  homomorphism of (not necessarily unital) $k$-algebras. We will examine, when $\varphi$ is spectrum preserving. A description of simple $\widetilde{M_n(B)}$-modules in terms of the simple $M_n(B)$-modules has been obtained in \cite[Theorem 3.10]{KXaffine} and will allow us to obtain the following correspondence between the primitive spectra of the two rings under the hypothesis that $\widetilde{M_n(B)}$ is idempotent. The correspondence from \cite{KXaffine} is revised in the proof of the following proposition for the sake of completeness.

Note that by \cite[Proposition 3.8]{KXaffine},
any left $M_n(B)$-module $V$ carries a natural structure of left
$\widetilde{M_n(B)}$-module, by defining 
$$ \tilde{a}\cdot v :=  \varphi(a)v = a\psi v, \qquad \forall \tilde{a}\in \widetilde{M_n(B)}, \: v\in V.$$
We denote this left $\widetilde{M_n(B)}$-module structure on $V$ by $V_\varphi$.

\newcommand{\tR}{?}

\begin{proposition} \label{primitives} Let $B$ be an affine $k$ algebra, $\psi\in M_n(B)$ and $\widetilde{M_n(B)}=(M_n(B), \psi)$. Assume $\widetilde{M_n(B)}$ is idempotent. There is a one-to-one correspondence between the primitive spectrum of $M_n(B)$ and the one of $\widetilde{M_n(B)}$ given by
$$\begin{array}{ccc}
\varphi^* : \Prim(M_n(B)) & \longrightarrow & \Prim(\widetilde{M_n(B)})\\[0.3cm]
P & \to & \{X\in M_n(B) : \psi X \psi \in P\}
\end{array}$$
\end{proposition}

\begin{proof}
Let $P\in \Prim(M_n(B))$. Then $P$ is a maximal ideal of $M_n(B)$ and $P=M_n(\mathfrak{m})$ for some ${\mathfrak m}$ a maximal ideal of $B$.

Let $V$ be a simple $M_n(B)$-module such that $P=\ann_{M_n(B)}(V)$. As above,
$V_{\varphi}$ is a $\widetilde{M_n(B)}$-module.
Note that $\ann_{V_{\varphi}}( \mathbf{I}_n)=\ann_V(\psi)$ is an $\widetilde{M_n(B)}$-submodule of $V_{\varphi}$ and consider $V_{\varphi}/ \ann_{V_{\varphi}}( \mathbf{I}_n )$. If $V_{\varphi}= \ann_{V_{\varphi}}(\mathbf{I}_n)$, then $\psi V=0$. As $\widetilde{M_n(B)}$ is idempotent, $M_n(B)\psi M_n(B)=M_n(B)$. As $V$ is a simple $M_n(B)$-module, it would follow that $V=M_n(B)V=M_n(B)\psi V=0$, a contradiction. Therefore $V_{\varphi}\neq \ann_{V_{\varphi}}(\mathbf{I}_n)$. Also $\widetilde{M_n(B)}( V_{\varphi}/ \ann_{V_{\varphi}}(\mathbf{I}_n))\neq 0$. Now take $W$ an $\widetilde{M_n(B)}$-submodule of $V_{\varphi}$ properly containing $\ann_{V_{\varphi}}(\mathbf{I}_n)$ and $w\in W\backslash \ann_{V_{\varphi}}(\mathbf{I}_n).$ We have that $M_n(B)\psi w\neq 0$ is a $M_n(B)$-submodule of $V$, hence $M_n(B)\psi w=V$ and $V_{\varphi}=V=M_n(B)\psi w=\widetilde{M_n(B)} \cdot w\subseteq W$. Hence 
$$V_{\varphi}/ \ann_{V_{\varphi}}(\mathbf{I}_n)$$
is a simple $\widetilde{M_n(B)}$-module and  
\begin{eqnarray*}\ann_{\widetilde{M_n(B)}}\left(V_{\varphi}/ \ann_{V_{\varphi}}(\mathbf{I}_n)\right)
&=&\{X\in M_n(B): X\psi v\in \ann_{V_{\varphi}}(\mathbf{I}_n),  \forall v\in V\}\\
&=&\{X\in M_n(B): \psi X\psi v=0, \forall v\in V\}.\end{eqnarray*}
Therefore 
\begin{equation}\label{UFF}
\ann_{\widetilde{M_n(B)}}\left(V_{\varphi}/ \ann_{V_{\varphi}}(\mathbf{I}_n)\right)= \{X\in M_n(B) : \psi X \psi \in P\}
\end{equation}
is a primitive ideal of $\widetilde{M_n(B)}$.


Let $Q\in \Prim(\widetilde{M_n(B)}$ and $E$ a simple left $\widetilde{M_n(B)}$-module such that $Q=\ann_{\widetilde{M_n(B)}}(E)$.  Consider the left $\widetilde{M_n(B)}$-module structure $M_n(B)_\varphi$ on $M_n(B)$. As $E$ is simple, $E=\widetilde{M_n(B)} e$ for any $e\in E\backslash\{0\}$. As in \cite[Proposition 3.10]{KXaffine} consider $q:M_n(B)_\varphi \rightarrow E$ such that $q(r)=\tilde{r}e$. Recall, that the product in $\widetilde{M_n(B)}$ is defined as $\tilde{a}\tilde{b}=\widetilde{a\psi b}$, for any $\tilde{a}, \tilde{b}\in \widetilde{M_n(B)}$. 
Given $\tilde{a}\in \widetilde{M_n(B)}, b\in M_n(B)$
$$\tilde{a}q(b)=\tilde{a}(\tilde{b}e)=(\tilde{a}\tilde{b})e=(\widetilde{a\psi b}) e=q(a\psi b)=q(\tilde{a}\cdot b).$$
So $q$ is a $\widetilde{M_n(B)}$-epimorphism and as $E$ is simple, $\Ker{q}$ is a maximal $\widetilde{M_n(B)}$-submodule of ${M_n(B)}_\varphi$. Let $Z$ be the centre of $M_n(B)$ and consider $Z\Ker{q}$ that contains $\Ker{q}$. Note that 
$$\tilde{a}\cdot (zx)=a\psi zx=z(\tilde{a}\cdot x)\in Z \Ker{q}, \qquad \forall a\in M_n(B), z\in Z, x\in \Ker{q}.$$
Hence $Z\Ker{q}$ is an $\widetilde{M_n(B)}$-submodule of $M_n(B)_\varphi$ that contains the maximal submodule $\Ker{q}$. Thus $\Ker{q}=Z\Ker{q}$ or $Z\Ker{q}=M_n(B)$. Assume  $Z\Ker{q}=M_n(B)$, then
$$E = q(M_n(B)) = q(\widetilde{M_n(B)}\cdot M_n(B))=q(M_n(B)\psi Z\Ker{q})=\widetilde{M_n(B)Z}q(\Ker{q})=0$$
a contradiction. Therefore,  $\Ker{q}=Z\Ker{q}$ and $\Ker{q}$ is a $Z$-module, as well as $E\cong M_n(B)/\Ker{q}$. As a $Z$-module, $E$ is  finitely generated and for every $\mathfrak{m}\in \MaxSpec(Z)$, $\mathfrak{m}E$ is an $\widetilde{M_n(B)}$-module, so $\mathfrak{m}E=0$ or $\mathfrak{m}E=E$.
If  $\mathfrak{m}E=E$ for all $\mathfrak{m}\in \MaxSpec(Z)$ using localisation and Nakayama's Lemma, we would have $E=0$ \cite[Proposition 3.8]{Atiyah}. Hence there is $\mathfrak{m}$, a maximal ideal of $Z=Z(M_n(B))\cong B$, such that $\mathfrak{m}E=0$. Thus $M_n(\mathfrak{m})\subseteq \Ker{q}$ and $q$ induces a $Z$-epimorphism $q':M_n(B)/M_n(\mathfrak{m}) \longrightarrow E$. Also $q'$ is a $\widetilde{M_n(B)}$-homomorphism from $(M_n(B)/M_n(\mathfrak{m}))_{\varphi}$ to $E$. Note that $\mathfrak{m}$ under the above conditions is unique.

As an $M_n(B)$-module, $M_n(B)/M_n(\mathfrak{m})\cong V_1\oplus\cdots\oplus V_n$  a direct sum of copies of a simple $M_n(B)$-modules with annihilator $M_n(\mathfrak{m})$. Since $q'\neq 0$, there is $i\in\{1,\ldots, n\}$ such that $q'(V_i)\neq 0$. For such an $i$, $\Ker{q'|_{(V_i)_{\varphi}}}=\ann_{(V_i)_{\varphi}}(\mathbf{I}_n)$, indeed if $x\in \Ker{q'}\cap (V_i)_{\varphi}$ and $\psi x\neq 0$, $M_n(B)\psi x=V_i$, for $V_i$ is a simple $M_n(B)$-module, and so $\widetilde{M_n(B)} x=(V_i)_{\varphi}$ and $q'((V_i)_{\varphi}) = q'(\widetilde{M_n(B)} x)=\widetilde{M_n(B)} q'(x)=0$, a contradiction. Thus $\Ker{q'}\cap V_i \subseteq \ann_{(V_i)_{\varphi}}(\mathbf{I}_n)$. If $x\in V_i$ is such that $\psi x=0$ then $\widetilde{M_n(B)} q'(x)=q'(\widetilde{M_n(B)} x)=q'(M_n(B)\psi x)=0$. Since $E$ is a simple $\widetilde{M_n(B)}$-module, if $q(x)\neq 0$,  $\widetilde{M_n(B)} q(x)$ would be $E$, hence $x\in \Ker{q'}$. It follows that, as an $\widetilde{M_n(B)}$-module, $E$ is isomorphic to $(V_i)_{\varphi}/ \ann_{(V_i)_{\varphi}}(\mathbf{I}_n)$, so by (\ref{UFF})
$$Q=\ann_{\widetilde{M_n(B)}}(E)=\ann_{\widetilde{M_n(B)}}((V_i)_{\varphi}/ \ann_{(V_i)_{\varphi}}(\mathbf{I}_n))=\{X\in M_n(B): \psi X\psi \in M_n(\mathfrak{m})\}.$$

Let $P_1, P_2$ be distinct elements of $\Prim(M_n(B))$. Since $P_1, P_2$ are maximal ideals, $P_1+P_2=M_n(B)$. As $P_1\subseteq \varphi^*(P_1)$ and $P_2\subseteq \varphi^*(P_2)$, if $\varphi^*(P_1) = \varphi^*(P_2)$, $M_n(B)=P_1+P_2\subseteq \varphi^*(P_1)$ and $\psi M_n(B) \psi \subseteq P_1$. By hypothesis $\widetilde{M_n(B)}$ is idempotent, so $M_n(B)\psi M_n(B)=M_n(B)$  and 
$M_n(B)=M_n(B)\psi M_n(B)\psi M_n(B)\subseteq P_1$ a contradiction, so $\varphi^*$ is injective.
\end{proof}

\begin{proposition}\label{prop:weakly}
The ring homomorphism $\varphi:\widetilde{M_n(B)}\to M_n(B)$ is spectrum preserving if and only if $\widetilde{M_n(B)}$ is idempotent.
 \end{proposition}

 \begin{proof}
Assume first that $\widetilde{M_n(B)}$ is not idempotent. It follows by \cite[Proposition 2.5(6)]{CKLS} that there is a maximal ideal $\mathfrak{m}$ of $B$ containing the ideal generated by all entries of $\psi$. Considering $M_n(\mathfrak{m})$, maximal ideal of $M_n(B)$, hence primitive, it follows that $\varphi(\widetilde{M_n(B)}) = M_n(B)\psi \subseteq M_n(\mathfrak{m})$, or in other words, $\varphi^{-1}(M_n(\mathfrak{m}))=\widetilde{M_n(B)}$. Thus, $\varphi^{-1}(M_n(\mathfrak{m}))$  is not contained in any primitive ideal of $\widetilde{M_n(B)}$ and $\varphi$ cannot be spectrum preserving, proving the only if part of the statement.

Assume now that $\widetilde{M_n(B)}$ is idempotent, take $P\in \Prim(M_n(B))$ and consider the bijection $\varphi^*$ between $\Prim(M_n(B))$ and $\Prim(\widetilde{M_n(B)})$, as in Proposition \ref{primitives}. By definition
$$P\subseteq \varphi^{-1}(P)=\{ X\in M_n(B) : X\psi \subseteq P\}\subseteq \varphi^*(P).$$
We need to show that $\varphi^*(P)$ is the unique primitive ideal containing $\varphi^{-1}(P)$. Suppose there exists a primitive ideal $Q\in \Prim(\widetilde{M_n(B)})$, such that $\varphi^{-1}(P)\subseteq Q$. By Proposition \ref{primitives}, $Q=\varphi^*(P_1)$, for some $P_1\in \Prim(M_n(B))$. In particular, $P_1\subseteq \varphi^{-1}(P_1) \subseteq Q$. Hence,  $P+P_1 \subseteq \varphi^{-1}(P) + \varphi^{-1}(P_1) \subseteq Q$. If $P\neq P_1$, then as both are maximal ideals, $P+P_1=M_n(B)$ and hence $Q=M_n(B)$, which is a contradiction. Thus  $\varphi^*(P)$ is the only primitive ideal that contains $\varphi^{-1}(P)$.
\end{proof}

\begin{definition}[\cite{BN}]
A morphism $\Theta: R\to S$ of $k$-algebras is called \emph{weakly spectrum preserving} if, and only if, there exist increasing filtrations
$0 = J_0 \subset J_1 \subset \cdots \subset J_n = R$ and 
$0 = I_0 \subset I_1\subset \cdots \subset I_n = S$ 
of two-sided ideals such that $\Theta(J_i)\subseteq I_i$ and the induced morphisms $J_i/J_{i-1} \to I_i/I_{i-1}$ are spectrum preserving.     
\end{definition}

Let $A$ be an affine cellular algebra $A$ with cell chain $0=J_0 \subset J_1 \subset \cdots \subset J_n=A$, such that each affine cell ideal $J_k/J_{k-1}$ has cell datum $(B_k, \sigma_k, \Delta_k,  \psi^{(k)},\alpha_k)$, i.e. $J_k/J_{k-1}\cong \Delta_k \otimes_{B_k}  \Delta_k' \cong \left( M_{d_k}(B_k), \psi^{(k)} \right)$ as $A$-bimodules, where $d_k=\rank_{B_k}(\Delta_k)$. Let $I_k = \asym_n \times   \cdots \times \asym_k$. Suppose no determinant $\det(\psi^{(k)})$ is a zero-divisor, then by Theorem \ref{TheoremAtext}, $\Theta$ is an embedding of $A$ into its asymptotic algebra $\asym = M_{d_n}(B_n)\times \cdots \times M_{d_1}(B_1)$. 
Using Proposition \ref{prop:weakly} we conclude
\begin{theorem}[Theorem \ref{TheoremC}] \label{TheoremCtext}
 Let $A$ be an affine cellular algebra with no determinant $\det(\psi^{(i)})$ being a zero-divisor. Then the embedding $\Theta:A\to \asym$ is weakly spectrum preserving if and only if all affine cell ideals $J_i/J_{i-1}$ are idempotent.
 \end{theorem}

 Now, \cite[Theorem 9]{BN} implies invariance of periodic cyclic homology
 under a general assumption in \cite{BN}. 
 
\begin{corollary}\label{pch}
  Let $A$ be an affine cellular algebra with no determinant $\det(\psi^{(i)})$
  being a zero-divisor and all $J_i/J_{i-1}$ being idempotent.
  If  $A$ is
  a finitely generated algebra over an affine $\mathbb{C}$-algebra, then the
  embedding $\Theta$ induces
 an isomorphism of periodic cyclic homology. 
\end{corollary}

In \cite{BN}, algebras with these properties are called finite type
algebras. Iwahori-Hecke algebras satisfy this conditions, and in \cite{BN}
the periodic cyclic homology of Iwahori Hecke algebras with parameter
not being zero and not a proper root of unity is determined by
using the embedding $\Theta$ into the asymptotic algebra and its property
of being weakly spectrum preserving. Cellular structures are not used in
\cite{BN}. Corollary \ref{pch} also applies to Brauer algebras and more
generally diagram algebras that are cellular in the sense of
\cite{GrahamLehrer} over an affine ground ring that is a $\mathbb{C}$-algebra.


\section{Examples and methods}

This section aims at relating the general theory of affine cellular
algebras to classes of examples occurring in applications, in particular
to algebraic Lie theory. This will happen in several respects.
First, an explicit example illustrating the concepts will get worked out
in detail.  Then we will identify our abstractly obtained embedding of $A$ into $\hat{A}$ with
Lusztig's embedding in the case of extended affine Hecke algebras of type $A$.  And finally we will summarise methods used in applications to compute Gram determinants, which
are determinants of swich matrices. This shows that in situations of
interest, these determinants often - under mild assumptions even generically -
are not zero-divisors, thus justifying assumptions of the previous
sections. 

\subsection{An explicit example}
\newcommand{\gs}{s_1}
\newcommand{\gt}{s_2}
\newcommand{\Cs}{C_{\gs}}
\newcommand{\Ct}{C_{\gt}}
\newcommand{\Cst}{C_{\gs\gt}}
\newcommand{\Cts}{C_{\gt\gs}}
\newcommand{\Csts}{C_{\gs\gt\gs}}

\newcommand{\Ts}{t_{\gs}}
\newcommand{\Tt}{t_{\gt}}
\newcommand{\Tst}{t_{\gs\gt}}
\newcommand{\Tts}{t_{\gt\gs}}
\newcommand{\Tsts}{t_{\gs\gt\gs}}

This section starts by providing a concrete example of the Hecke algebra $H$
of type $A_2$, which is cellular over $R=\ZZ[q^{\pm 1/2}]$,  i.e.\!\!\! of the
symmetric group $W=\Sigma_3 = \{ id, \gs, \gt, \gs\gt,\gt\gs,\gs\gt\gs \}$,
where $\gs=(12)$ and $\gt=(23)$. The Kazhdan-Lusztig basis of $H$ is
$\{C_{id}, \Cs, \Ct, \Cst, \Cts, \Csts \}$
with $C_{id}$ being the identity element.
There are exactly three two-sided cells, $\mathcal{C}=\{ \{C_{id}\},
\{\Cs,\Ct,\Cst,\Cts\}, \{\Csts\}\}$ and the cellular structure of $H$ is given
by $H=J_1\oplus J_2' \oplus J_3'$, where
$$J_1 = R\Csts, \qquad J_2' = R\Cs\oplus R\Ct \oplus R\Cst \oplus R\Cts,
\qquad J_3'= RC_{id}.$$
The cell chain is given by $J_1 \leq J_2:=J_1\oplus J_2' \leq J_3:=
J_2\oplus J_3'=H$.
The multiplication of the Kazhdan-Lusztig basis is determined by $ C_w C_u =
\sum_{v\in W} h_{w,u,v} C_v$
for elements $h_{w,u,v}\in R$ given by the following multiplication table
(omitting products with the identity $C_{id}$), where we set
$\eta := -(q^{1/2}+q^{-1/2}).$
\begin{figure}[h]
$$\begin{array}{|c||c|c|c|c|c|}\hline
     &   \Cs        & \Ct          & \Cst          & \Cts                  & \Csts\\\hline\hline
\Cs  &  \eta \Cs     & \Cst         & \eta \Cst      & \Cs + \Csts           & \eta \Csts \\\hline
\Ct  &  \Cts        & \eta\Ct      & \Ct + \Csts       & \eta \Cts              & \eta \Csts \\\hline
\Cst &  \Cs + \Csts & \eta\Cst     & \Cst + \eta \Csts  & \eta \Cs + \eta \Csts   & \eta^2 \Csts \\\hline
\Cts & \eta\Cts     & \Ct + \Csts & \eta \Ct + \eta \Csts & \Cts + \eta \Csts & \eta^2 \Csts \\\hline
\Csts & \eta \Csts   & \eta\Csts  & \eta^2 \Csts  & \eta^2\Csts    & (\eta^3-\eta) \Csts \\\hline
\end{array}$$
\caption{Multiplication table for the KL-basis element of the  Hecke algebra of $\Sigma_3$}\label{KL-mult}
\end{figure}

\medskip

The basis elements of $J_2/J_1$ can be written (and ordered) as the cosets
$E_{11} = \Cs + J_1$, $E_{21} = \Cts + J_1$, $E_{12} = \Cst + J_1$,
$E_{22}=\Ct + J_1$ and the multiplication in $J_2/J_1$ is given by 
\[
\begin{array}{|c||c|c|c|c|}\hline
     &   E_{11}        & E_{22}          & E_{12}          & E_{21}                 \\\hline\hline
E_{11}  &  \eta E_{11}     & E_{12}         & \eta E_{12}      & E_{11}             \\\hline
E_{22}  &  E_{21}        & \eta E_{22}      & E_{22}   & \eta E_{21}               \\\hline
E_{12} &  E_{11} & \eta E_{12}     & E_{12}   & \eta E_{11}    \\\hline
E_{21} & \eta E_{21}     & E_{22} & \eta E_{22}  & E_{21}   \\\hline
\end{array}
\]
Identifying $E_{ij}$ with the matrix units of $M_2(R)$, we see that $J_2/J_1
\cong  (M_2(R),\psi)$, where
$\psi = \left(\begin{array}{cc}
                     \eta & 1 \\ 1 & \eta
              \end{array}\right)$. Note that the determinant of $\psi$
            is $\eta^2-1$, hence non-zero and $1$ belongs to the ideal
            generated by the entries of $\psi$, hence $J_2/J_1$ is an
            idempotent ideal of $H/J_1$.
Interestingly, the swich matrix of $J_1$ is $(\eta^3-\eta)$, which is 
not a zero-divisor, but also not invertible, hence $J_1$ is not idempotent in $H$.

Our embedding $H \to \widehat{H} = R \times M_2(R) \times R$ can be calculated
iteratively.
First we consider $$\Theta_1: H \to H/J_1 \times \mathrm{End}((J_1)_{H}) \to
H/J_1 \times R,$$ by sending an element $h\in H$ to $(h+J_1, \lambda_h^1)$,
where $\lambda_h^1$ denotes the left multiplication by $h$ on $J_1$. The
isomorphism $\mathrm{End}({(J_1)}_{H})\cong R$ is obtained by using the
isomorphism $\alpha_1:J_1\to R$ given by $\alpha_1(\Csts) = 1$. Then an
endomorphism $f\in \mathrm{End}({(J_1)}_{H})$ is mapped to
$\alpha_1(f(\alpha_1^{-1}(1))) = \alpha_1(f(\Csts))$. Combined this gets us to
$$\Theta_1(h) = (h+J_1, \alpha_1(h\Csts)), \qquad \forall h\in H.$$
So the values of the second component of $\Theta_1(C_{w})$ are given by the
scalars $1$, $\eta$, $\eta^2$ or $\eta^3-\eta$ of the last column of the
multiplication table (Figure \ref{KL-mult}).
Considering $H/J_1$ we have the embedding
$$\Theta_2: H/J_1 \to H/J_2 \times \mathrm{End}((J_2/J_1)_{H/J_1}) \to R
\times M_2(R).$$
For an element $E_{ij}$ of the basis of $J_2/J_1$ the embedding in the second
component is given by multiplication with $\psi$, which we can write
symbolically as
$$\left(\begin{array}{cc}
         E_{11} & E_{12} \\ E_{21} & E_{22}
        \end{array}\right)
        \mapsto
        \left(\begin{array}{cc}
         E_{11} & E_{12} \\ E_{21} & E_{22}
        \end{array}\right)
        \left(\begin{array}{cc}
         \eta & 1 \\ 1 & \eta
        \end{array}\right)
        =
        \left(\begin{array}{cc}
         \eta E_{11} + E_{12} & E_{11} + \eta E_{12} \\ \eta E_{21}+E_{22} & E_{21}+\eta E_{22}
        \end{array}\right),$$
      where each entry on the left  is mapped to the corresponding entry on
      the right. Since $C_{id}$ is the identity element, it will be mapped to the
      identity matrix. Altogether we have obtained a map 
        \begin{equation}
          \Theta = (\Theta_2 \times id)\circ \Theta_1: H \to \widehat{H} =
          R \times M_2(R) \times R
        \end{equation} given by
\begin{eqnarray*}
 \Theta(C_{id})  &=& \left( 1, E_{11}+E_{22}, 1 \right) \\
 \Theta(\Cs)  &=& \left( 0, \eta E_{11} + E_{12}, \eta \right)\\
 \Theta(\Ct)  &=& \left( 0, E_{21}+\eta E_{22}, \eta \right) \\
 \Theta(\Cst)  &=& \left( 0, E_{11} + \eta E_{12}, \eta^2 \right) \\
 \Theta(\Cts)  &=& \left( 0, \eta E_{21}+E_{22}, \eta^2 \right) \\
 \Theta(\Csts)  &=& \left( 0, \mathbf{0}, \eta^3-\eta \right) \end{eqnarray*}

In comparison, we will explicitly determine Lusztig's embedding $\varphi$.
According to \cite[p.538]{Lusztig2}, the $a$-function is determined by the
multiplication table in Figure \ref{KL-mult}. For $v\in W$ one has
$$a(v) = \min \left\{ i\geq 0 :   (-q^{1/2})^i \: h_{w,u,v}
  \in R^+,  \:  \: \forall w,u\in W  \right\},$$
where $R^+ = \Z[q^{1/2}]$.
Note that according to Figure \ref{KL-mult},  $h_{w,u,v} \in \{0,1,\eta,\eta^2,
\eta^3-\eta\}$ and that the least $i\geq 0$ that puts $\eta, \eta^2,
\eta^3-\eta$  into $R^+$ are $1,2,3$ respectively.
Hence looking at table \ref{KL-mult}, we see
$$a(w)=1,\: \forall w\neq \{id, \gs\gt\gs\}, \qquad  a(\gs\gt\gs)=3, \qquad
a(id)=0.$$
According to \cite[p.538]{Lusztig2}, the integers  $\gamma_{w,u,v}$ are
determined by
$$ \gamma_{w,u,v} = \left(-q^{1/2}\right)^{a(v)} h_{w,u,v^{-1}} \:
(\mathrm{mod}\: q^{1/2}R^+)$$
If $h_{w,u,v^{-1}}\in\{0,1\}$, then $\gamma_{w,u,v}=0$.  In particular, for
$v\neq \gs\gt\gs$,
$\gamma_{w,u,v} = 1$ if and only if $h_{w,u,v}=\eta$, otherwise it is zero. For
$v=\gs\gt\gs$, $\gamma_{w,u,v} = 1$ if and only if $w=u=\gs\gt\gs$, otherwise
it is zero. Since the element $C_{1}$ acts as an identity element,
$h_{id,u,v}= \delta_{u,v}$, $h_{w,id,v} = \delta_{w,v}$ and $h_{w,u,id}=
\delta_{w,id}\delta_{u,id}$.
According to \cite[2.3]{Lusztig2}, the \emph{base ring} $B$ is defined as the
free $\Z$-module with basis $\{t_w : w\in W\}$ with multiplication $ t_wt_u =
\sum_{v\in W} \gamma_{w,u,v} t_{v^{-1}}.$
Hence we obtain the following multiplication table of $B$:
$$\begin{array}{|c||c||c|c|c|c||c|}\hline
          &  t_{id} & \Ts        & \Tt          & \Tst          & \Tts                  & \Tsts\\\hline\hline
t_{id}    & 1  & 0        & 0          & 0          & 0                  & 0\\\hline\hline
     \Ts  &  0 &\Ts     & 0         & \Tst      & 0           & 0 \\\hline
\Tt  &  0& 0        & \Tt      & 0   & \Tts              & 0 \\\hline
\Tst &  0&0 & \Tst     & 0  & \Ts   & 0 \\\hline
\Tts & 0&\Tts     & 0 & \Tt  & 0 & 0 \\\hline\hline
\Tsts &0& 0   & 0  & 0  & 0    & \Tsts \\\hline
\end{array}$$

This shows, $B \cong \Z \times M_2(\Z) \times \Z$ and
$B\otimes_{\mathbb{Z}} R = \widehat{H}$, where 
$$t_{id} \mapsto (1, \mathbf{0}, 0), \qquad \Tsts \mapsto (0,\mathbf{0},1)$$
$$ \Ts \mapsto (0,E_{11},0), \qquad  \Tt \mapsto (0,E_{22},0), \qquad \Tst
\mapsto (0,E_{12},0), \qquad  \Tts\mapsto (0,E_{21},0).$$
The distinguished involutions $D$ of $W$ are defined as $D = \{ u\in W :
a(u)=l(u)-2\delta(u)\}$, where $\delta(u)$ is the degree of the KL-Polynomal
$P_{e,u}$ as a polynomial in $q^{-1/2}$. In our case, all KL-polynomials are
constant, as for all dihedral groups, see \cite[7.12(a)]{humphreys}.
Hence $\delta(u)=0$ and using the values of the $a$-function, we
obtain 
$$  D = \{ u\in \Sigma_3 : a(u)=l(u)\} = \{ id, \gs, \gt, \gs\gt\gs\}.$$
Following \cite[2.4]{Lusztig2}, the embedding of $H$ into $B$ is defined as
the left $R$-linear map $\varphi:H \to B\otimes_{\mathbb{Z}} R$ determined by 
$$\varphi(C_w) = \sum_{d\in D, v\in W, a(d)=a(v)} h_{w,d,v}\: t_v,$$
where the sum runs over precisely the pairs of the set
$$\{ (id,id), (\gs,\gs),(\gs,\gt), (\gs,\gs\gt), (\gs, \gt\gs), (\gt,\gs), (\gt,\gt), (\gt, \gs\gt), (\gt,\gt\gs), (\gs\gt\gs, \gs\gt\gs)\}.$$
Going through the multiplication table in Figure \ref{KL-mult}, we find by inspection $\varphi = \Theta$:
\begin{eqnarray*}
\varphi(C_{id}) &=& t_{id} + \Ts + \Tt + \Tsts \mapsto (1, E_{11}+E_{22}, 1)\\
\varphi(C_{\gs}) &=& \eta  \Ts + \Tst + \eta \Tsts \mapsto (0, \eta E_{11} + E_{12}, \eta)\\
\varphi(C_{\gt}) &=& \eta  \Tt + \Tts + \eta \Tsts \mapsto (0, \eta E_{22} + E_{21}, \eta)\\
\varphi(C_{\gs\gt}) &=&   \Ts + \eta \Tst + \eta^2 \Tsts \mapsto (0, E_{11} + \eta E_{21}, \eta^2)\\\
\varphi(C_{\gt\gs}) &=&   \Tt + \eta \Tts + \eta^2 \Tsts \mapsto (0, E_{22} + \eta E_{21}, \eta^2)\\\
\varphi(C_{\gs\gt\gs}) &=&   (\eta^3-\eta) \Tsts  \mapsto (0, \mathbf{0}, (\eta^3-\eta))\\
\end{eqnarray*}

\subsection{Extended affine Hecke algebra of type \emph{A}}

An extended affine Hecke algebra $H$ of type $A_{n-1}$ is an affine cellular algebras, as shown in 
\cite[Section 5]{KXaffine}. 
 We will briefly revise its affine cellular structure. Let $W$ be an extended Weyl group of type $A_n$ (with Coxeter system $(W,S)$) and let $\{C_w:w\in W\}$ be the Kazhdan-Lusztig basis for the Hecke algebra $H$ of $W$ over $R=\Z[q^{\pm 1/2}]$ as in \cite{KazhdanLusztig}. Then 
\begin{equation}
C_wC_u=\sum_{z\in W}h_{w,u,z}C_z
\end{equation}
for some $h_{w,u,z}\in R$.  Let ${\mathcal L}$ be Lusztig's asymptotic Hecke algebra of $(W,S)$ as in \cite{Lusztig2}. This is the free $\Z$-algebra with basis $\{t_w:w\in W\}$ and multiplication given by
\begin{equation}
t_wt_u=\sum_{z\in W} \gamma_{w,u,z^{-1}}t_z
\end{equation}
for some $\gamma_{w,u,z}$.
Let ${\cB}=R\otimes_{\Z} {\mathcal L}$ be Lusztig's asymptotic Hecke algebra with coefficients in $R$.  The \emph{asymptotic Hecke algebra of a cell $c$} is the free $R$-submodule $\cB_c$ of $\cB$ generated by $t_w$, for  $w\in c$, which is a two-sided ideal in $\cB$ with identity $\sum_{d\in D_c} t_d$, 
where $D$ is the set of distinguished involutions of $W$ and 
$D_c = D \cap c$.
Moreover, $\cB$ is a direct product of the $\cB_c$'s, where $c$ ranges over all two-sided cells of $W$ (see \cite[p.175]{KXaffine} or  \cite{Xi}).
By \cite[8.2]{Xi}, ${\cB}_c$ is isomorphic to a matrix ring with entries in a commutative $\Z$-affine algebra.

For a two-sided cell $c$ consider $H_{\leq c}$, the subspace of $H$ generated by $C_w$ with $w\in c' \leq c$. Similarly, define 
$H_{<c}$ and $H_c=H_{\leq c}/H_{<c}$, which is a (possibly non-unital) $R$-algebra with basis $\{[C_w]: w\in c\}$ and multiplication given by
\begin{equation}
[C_w][C_u]=\sum_{z\in c}h_{w,u,z}[C_z]
\end{equation}

As $\left\{[C_w]: w\in c\right\}$ is an $R$-basis of $H_c$, there is a bijection $\alpha_c: H_c \longrightarrow {\cB}_c$ such that $\alpha_c([C_w])=t_w$, which extends to an $(R,R)$-isomorphism. Thus, the $H$-bimodule structure of $H_c$ and $\alpha_c$ induce an $H$-bimodule structure on $B_c$ as follows: for any  $u\in W, w\in c$ we set
\begin{equation}
 t_w\cdot C_u := \alpha_c([C_wC_u]) = \sum_{z\in c} h_{w,u,z} \alpha_c([C_z]) = \sum_{z\in c} h_{w,u,z} t_z.
\end{equation}
Similarly, $C_u\cdot t_w :=\sum_{z\in c} h_{u,w,z}t_z$.
Hence, $\alpha_c$ is an $(H,H)$-bimodule homomorphism and we conclude, using \cite[Lemma 5.5]{KXaffine}, for  $u, w\in c$:
\begin{equation}\label{eq:hecke-mult-cell}
\alpha_c([C_w][C_u]) 
= t_{w}\cdot[C_u]
= t_w\left(\sum_{d\in D_c} t_d \cdot \sum_{d\in D_c} [C_d]\right)t_u.
\end{equation}
Define $\psi^c := \sum_{d\in D_c} t_d \cdot \sum_{d\in D_c} [C_d]\in B_c$, then for any $w,u\in c$:
\begin{equation}\label{eq:mult-c-t}\alpha_c([C_w][C_u]) = t_w\psi^c t_u = \alpha_c([C_w])\psi^c \alpha_c([C_u]).\end{equation}
Hence $H_c\cong (\cB_c,\psi^c)$ with $\cB_c$ being a matrix ring over an affine commutative ring.

Following \cite{Xi}, Lusztig's embedding is defined as follows:

\begin{definition}
  The injective $R$-algebra homomorphism,
$$\varphi: H \to \cB, \qquad  C_w \mapsto  \sum_{z\in W} \sum_{\substack{d\in D\\ a(z)=a(d)}} h_{w,d,z}t_z$$
is called {\em Lusztig's embedding}.
Here, $a$ is Lusztig's $a$ function $a:W\to \Z$ (see \cite{Xi}).
\end{definition}

Consider $\pi_c:\cB \to {\cB}_c$,  the projection onto  the $c$-component of $\cB$, and the  $R$-algebra homomorphism $\pi_c\varphi: H\to {\cB}_c$. Since $\pi_c \varphi(H_{<c})=0$, there is a well-defined $R$-algebra homomorphism $\varphi_c: H_c \to  {\cB}_c$ given by \begin{equation}
\varphi_c\left([C_w]\right) =  \sum_{z\in c}\sum_{d\in D_c} h_{w,d,z}t_z,
\end{equation}
for $w\in c$.
Note that if $w,u\in c$, then $a(w)=a(u)$ \cite[Lemma 5.1]{KXaffine}, and furthermore, by  \cite[Lemma 5.1(3)]{KXaffine}, $z\leq_L d$, as $a(d)=a(z)$, $d{\mathtt{\sim}}_L z$.

The $H$-bimodule structure on $\cB_c$ yields also a $({\cB}_c, H_c)$-bimodule via
\begin{equation}
t_w\cdot [C_u] :=\sum_{z\in c} h_{w,u,z}t_z
\end{equation}
for $w,u\in c$, and we can rewrite $\varphi_c([C_w])$ as 
\begin{equation}\label{eq:varphi-formula}
\varphi_c([C_w])=t_w\cdot \sum_{d\in D_c} [C_d].
\end{equation}
Since $\sum_{d\in D_c} t_d$ is the identity in ${\cB}_c$, we have for any $w\in c$:
\begin{equation}
\varphi_c([C_w])=t_w\left(\sum_{d\in D_c} t_d \cdot \sum_{d\in D_c} [C_d]\right) = t_w \psi^c.
\end{equation}

\begin{lemma} \label{lemma:phic-injective}
Consider Lusztig's embedding $\varphi: H \to \cB$ and a two-sided cell $c$. Then the corresponding map 
$\varphi_c: H_c \to {\cB}_c$ given by 
$$[C_w]  \mapsto \displaystyle \sum_{d\in D_c} \sum_{z\in c}  h_{w,d,z}t_z = t_w \psi^c$$
is injective. In particular, $\psi^c$ is not a zero-divisor in $\cB_c$.
\end{lemma}
\begin{proof}
Recall, that  $R=\Z[u^{\pm 1}]$, for  $u=-q^{1/2}$. 
For any non-zero Laurent polynomial $p\in R$, there exists a smallest $b\in \ZZ$ such that 
\begin{equation}
u^b p \in \Z[u].
\end{equation}
Since $p\neq 0$, the constant term $(u^b p)(0)$ of $u^bp$ is non-zero.
According to \cite[p.538]{Lusztig2}, for any $z\in W$,  $a(z)$ is the smallest integer such that $u^{a(z)}h_{x,y,z}\in \Z[u]$, for all $x,y\in W$. Moreover, $\gamma_{x,y,z}$ is defined  as the constant term of $u^{a(z)}h_{x,y,z}$:
\begin{equation}\label{def:gamma}\gamma_{x,y,z} = \left(u^{a(z)}h_{x,y,z^{-1}}\right)(0)
\end{equation}

Suppose there exists a non-zero element  $g = \sum_{x\in c} p_x [C_x] \in \mathrm{\Ker{\varphi_c}}$. Let $b\in \Z$ be an integer, such that  $u^bp_x \in \Z[u]$, for all $x\in c$, and such that there exists at least one $x_0\in c$ such that $(u^bp_{x_0})(0)\neq 0$. Choose such element $x_0$ and set $\lambda_x=(u^bp_x)(0)$, for all $x\in c$.

Since $\varphi_c(g)=0$, equation (\ref{eq:varphi-formula}) implies:
\begin{equation}
0=\sum_{x\in c} p_x \varphi_c([C_x]) = \sum_{x\in c} p_x t_x\cdot \sum_{d\in D_c} [C_d] 
= \sum_{z\in c} \left( \sum_{x\in c}  \sum_{d\in D_c} p_x h_{x,d,z}\right) t_z.
\end{equation}
Since the elements $t_z$ are linearly independent over $R$, we have for  $z=x_0$:
\begin{equation}\label{eq:conclusion}
 \sum_{x\in c}  \sum_{d\in D_c} p_x h_{x,d,x_0} = 0.
\end{equation}
Note that the $a$-function is constant on the cell $c$. Hence there exists a number $a$ such that $a(x)=a$, for any $x\in c$. Multiplying equation (\ref{eq:conclusion}) by $u^{a+b}$ and evaluating at $0$, yields
\begin{equation}\label{eq:conclusion2}
 \sum_{x\in c}  \lambda_x \left(\sum_{d\in D_c} \gamma_{x,d,x_0^{-1}}\right) = 0,
\end{equation}
using $\lambda_x=(u^bp_x)(0)$ and   $\gamma_{x,d,x_0^{-1}} = (u^ah_{x,d,x_0})(0)$.
By \cite[page 543]{Lusztig2}, $\left(\sum_{d\in D_c} \gamma_{x,d,x_0^{-1}}\right)=1$ if and only if $x=x_0$ and otherwise $0$. Hence equation (\ref{eq:conclusion2}) implies $\lambda_{x_0}=0$, which contradicts the choice of $x_0$. Therefore, $\varphi_c$ must be injective and since $\varphi_c([C_w]) = t_w\cdot \psi^c$ and $\{[C_w] : w\in c\}$ is a basis of $H_c$, $\psi^c$ is not a zero-divisor.
\end{proof}

We will compare Lusztig's embedding with the embedding
obtained in Theorem \ref{TheoremAtext}, using Lemma
\ref{lemma:phic-injective}.

\begin{proposition} \label{embeddingscoincide}
  Let $H$ be the extended affine Hecke algebra of type $A_{n-1}$. Then Lusztig's embedding $\varphi:H\to \cB$ and the embedding $\Theta:H\to \cB$ of Theorem \ref{TheoremAtext} coincide.
\end{proposition}
\begin{proof}
Let $\cC$ be the set of all cells. For $c\in \cC$, $d\in D_c$  and  $w,z\in W$ we have by \cite[p.540]{Lusztig2} that  if $h_{w,d,z}\neq 0$ then $z\leq_L d$ (where $\leq_L$ denotes the left preorder on $W$ as defined in \cite{KazhdanLusztig}). If furthermore $a(d)=a(z)$, then by \cite[2.3.2]{GeckJacon}, we have $d\sim_L z$, which means that $z\in c$. Note also that the $a$-function is constant on cells. Hence we can refine Lusztig's embedding as

\begin{equation}\label{eq:littlePhi}
\varphi(C_w) = \sum_{c\in \cC}  \sum_{d\in D_c} \sum_{z\in c} h_{w,d,z} t_z
\end{equation}

The affine cell structure of $H$ is provided by the ideals $J_c = H_{\leq c}$, which are spanned by elements $C_w$, with $w\in c'\leq c$. The affine cell ideals are $J_c/J_{<c} =  H_{\leq c}/H_{< c} = H_c$ and we have seen that $\alpha : H_c \to \cB_c$, given by $[C_w] \mapsto t_w$ is an $(H,H)$-bimodule isomorphism. 

The embedding considered in Theorem \ref{TheoremAtext}
takes an element $C_w$ into the product of endomorphism rings of the affine cell ideals, i.e.
\begin{equation}
C_w \mapsto ( \lambda^c_{C_w}: H_c \to H_c  )_{c\in \cC} \in \prod_{c\in\cC} \mathrm{End}_H({H_c}),
\end{equation}
where for each $c\in \cC$ and $w\in W$ the map $\lambda^c_{C_w}$ represents the left action of $C_w$ on $H_c$.
We saw  that $\alpha([C_w][C_u]) = \alpha([C_w]) \psi^c \alpha([C_u])$ holds, for all $w,u \in c$, by
equation (\ref{eq:hecke-mult-cell}). In order to ease the notation we will simply write $\psi$ for $\psi^c$. 
The map $H_c \to \cB_c$ given by 
$[C_w] \mapsto \alpha([C_w]) \psi = t_w \psi$ is an injective ring homomorphism from $H_c$ to $\cB_c$ by Lemma \ref{lemma:phic-injective} and also a ring isomorphism between  $H_c$ and the generalised matrix ring $\widetilde{\cB_c}=(\cB_c,\psi)$.
This isomorphism of rings yields  an isomorphism of their endomorphism rings
\begin{equation}
\overline{\alpha}: \: \mathrm{End}_{H_c}({H_c}) \to \mathrm{End}_{\tilde{\cB_c}}({\tilde{\cB_c}}), \qquad f \mapsto \overline{\alpha}(f) :=  \alpha f \alpha^{-1},
\end{equation}
since for any $f \in \mathrm{End}_{H_c}({H_c})$ and $w,u\in c$:
\begin{eqnarray*}
\overline{\alpha}(f)(t_w \psi  t_u) = \alpha\left( f \left( [C_w][C_u] \right)\right)
= \alpha\left( f( [C_w] ) [C_u] \right)
= \alpha\left( f( [C_w] ) \right) \psi \alpha\left([C_u] \right)
= \overline{\alpha}(f)(t_w)\psi t_u
\end{eqnarray*}
holds, showing that $\overline{\alpha}(f)$ is right $\widetilde{\cB_c}$-linear. From \cite[Proposition 2.2]{CKLS} we know that the injective ring homomorphism $l: \cB_c \mapsto \mathrm{End}_{\tilde{\cB_c}}({\tilde{\cB_c}})$ given by $t_w \mapsto l(t_w): [t_u \mapsto t_wt_u]$ is an isomorphism, provided that $\psi$ is not a zero-divisor in $\cB_c$. In this case $l^{-1}: \mathrm{End}_{\tilde{\cB_c}}({\tilde{\cB_c}}) \to \cB_c$ is given by 
$l^{-1}(g)=g(1_{\cB_c})$, for any $g\in\mathrm{End}({\tilde{\cB_c}}_{\tilde{\cB_c}})$.
Since by Lemma \ref{lemma:phic-injective}, $\psi^c$ is not a zero-divisor, we conclude that 
\begin{equation}
l^{-1}\overline{\alpha}:\End_{H_c}({H_c}) \to \cB_c, \qquad f \mapsto \overline{\alpha}(f)(1_{\cB_c}),
\end{equation}
 is an isomorphism of rings. In particular, for  some left multiplication $f=\lambda_{C_w}$, we obtain
 \begin{equation}
 l^{-1}\overline{\alpha}(\lambda_{C_w}) 
 = \alpha\left( C_w \sum_{d\in D_c} [C_d]\right)
 =  \sum_{d\in D_c}  \alpha \left([C_wC_d]\right)
  =  \sum_{d\in D_c}  \sum_{z\in c} h_{w,d,z} t_z. 
 \end{equation}

Hence, we get the composed ring homomorphism

\begin{equation}\label{eq:bigPhi}
\Theta: H 
\to \prod_{c\in\mathcal{C}} \mathrm{End}_H({H_c})  
\to  \prod_{c\in\mathcal{C}} \mathrm{End}_{\widetilde{\cB_c}}(\widetilde{\cB_c}) 
\to  \cB, \qquad C_w \mapsto \Theta(C_w) = \sum_{c\in \mathcal{C}} \sum_{d\in D_c} \sum_{z\in c} h_{w,d,z} t_z. 
\end{equation}

Comparing (\ref{eq:littlePhi}) and (\ref{eq:bigPhi}), we see that $\varphi=\Theta$.
\end{proof}

\subsection{Gram determinants and Jucys Murphy elements}

In the case of extended affine Hecke algebras of type $A$, no determinant
of a swich matrix is a zero-divisor, by Lemma \ref{lemma:phic-injective} and
Proposition \ref{embeddingscoincide}. Thus the assumptions in our main
theorems are satisfied for this class of algebras. For other classes of
algebras these assumptions are known to be satisfied as well. We are
going to provide some references to methods used and results obtained.

In the classical situation of cellular algebras in \cite{GrahamLehrer}, Gram
determinants arise in the following way: A cellular algebra $A$ in
\cite{GrahamLehrer} is
free of finite rank over an integral domain $k$. The cell modules are denoted
by $\Delta_i$. 
The bilinear form $\psi^{(i)}:\Delta_i \times \Delta_i \to k$ is determined by
the matrix $\left( \psi^{(i)}(a_{s1}^i, a_{1t}^i) \right)_{1\leq s,t\leq d_i}$.
Its determinant is called the \emph{Gram determinant} of $\psi^{(i)}$ and
denoted by
$G(i) = \det\left( \psi^{(i)}(a_{s1}^i, a_{1t}^i) \right)_{1\leq s,t\leq d_i}$. 
By definition, these Gram determinants coincide with the determinants of
swich matrices when we view the cellular algebra $A$ as an affine cellular
algebra.

Consider $\Lambda = \{ (i,s) \mid 1\leq i \leq n, \: 1\leq s \leq d_i\}$ with lexicographic order $<$.  In \cite{Mathas}, Mathas introduced an abstract framework of \emph{Jucys-Murphy}-elements for a cellular algebra $A$ over an integral domain $k$. A family of commuting elements $\{L_1, \ldots, L_M\}$ of $A$ is called a family of \emph{$JM$-elements} if they are invariant with respect to the involution $i$ of the cellular algebra $A$, such that
\begin{equation}\label{eq:JM}
L_j a_{st}^i \equiv c_{i,s}(j) a_{st}^i + \sum_{v<s} r_{vt}(j) a_{vt}^i \: (\mathrm{mod} \: J_{i-1}),
\end{equation}
for all $1\leq i \leq n$, $1\leq s,t \leq d_i$ and $1\leq j \leq M$ holds, where $r_{vt}(j)$ and $c_{i,s}(j)$ are elements of $k$.
The $JM$-elements $\{L_1, \ldots, L_M\}$  generated a commutative subalgebra $\cL$ of $A$. From $(\ref{eq:JM})$ we deduce that for each pair $(i,s)\in \Lambda$, there exists a one-dimensional $\cL$-module  $K_{i,s}$, generated by $a_{s1}^i$ modulo $\sum_{v<s} ka_{v1}^i + J_{i-1}$, such that $L_j$ acts on $K_{i,s}$ by the scalar $c_{i,s}(j)$.

The JM-elements are said to fulfil the \emph{separation condition} (see \cite[Definition 2.8]{Mathas}),  if for all pairs $(i,s), (i',s') \in \Lambda$ with $(i,s)<(i',s')$, there exists $1\leq j \leq M$ such that $c_{i,s}(j)\neq c_{i',s'}(j)$.  In other words, $K_{i,s}\not\cong K_{i',s'}$ if  $(i,s)<(i',s')$. In case the JM-elements satisfy the separation condition, Mathas proved in \cite[Corollary 3.9, Theorem 3.12, Proposition 3.14]{Mathas} that the Gram determinant $G(i)$ is a non-zero element of $k$; hence not a zero-divisor in the integral domain $k$.

Suppose now that $A$ admits \emph{Jucys-Murphy}-elements $\{L_1, \ldots, L_M\}$ that satisfy the separation condition, then equation (\ref{eq:JM}) holds. Let $f\in \Hom_A(\Delta_j,\Delta_i)$ be any homomorphism with $i<j$. Passing from $k$ to its fraction field $F$, we have seen that $\Delta_i \otimes F$ has a composition series with composition factors $K_{i,s} \otimes F$. If $f\neq 0$, then $f$ is non-zero on a composition factor $K_{j,t} \otimes F$ of $\Delta_j$ and its image $f(K_{j,t}) \otimes F$ would be isomorphic to a composition factor $K_{i,s} \otimes F$, which is impossible, since $(i,s)<(j,t)$. Hence $\Hom_A(\Delta_j,\Delta_i)=0$, for $j>i$. Together with a result by Graham-Lehrer 
\cite[Proposition 2.6]{GrahamLehrer},
we have $\Hom_A(\Delta_j,\Delta_i)=0$, for $i\neq j$.
Mathas showed that the following cellular algebras over an integral domain $R$ admit JM-elements that satisfy the separation condition, namely the group algebras of symmetric groups, Hecke algebras of type $A$, the Ariki-Koike algebra at parameter  $q\neq 1$ as well as the degenerated Ariki-Koike algebra (see \cite[Theorem 4.10]{Mathas}).

Mathas showed in particular that the Hecke algebras
$H$ of type $A$ over $R=\Z[q^{\pm 1}]$
have Jucys-Murphy elements (see \cite[Example 2.15]{Mathas}) and that the
JM-elements satisfy the separation condition (even for those specialisations
of $q$ such that  $[1]_q[2]_q\cdots [n]_q \neq 0$, for
$[k]_q = 1+q+\cdots + q^{k-1}$, see \cite[Theorem 3.32]{Mathas99}). As
mentioned previously, cellular algebras with JM-elements that
satisfy the separation condition have swich matrices $\psi$ whose Gram
determinant is not a zero-divisor. This argument yields another proof
of the fact for determinants of $\psi^c$ matrices not to be zero-divisors.
\medskip

Using such methods, Gram determinants for several
families of cellular algebras over commutative rings have been determined
by explicit and usually recursive formulae. Examples of such results are
due to Rui and Si for Brauer algebras (\cite[Theorem 4.11]{RuiSibrauer})
and for Birman-Murakami-Wenzl algebras (\cite[Theorem 4.13]{RuiSibmw}) and
also for cyclotomic affine Hecke algebras of type $A$ by Hu and Mathas
(\cite[Theorem 3.13]{HuMathas}). Typically, the ground ring for the formulae
is a ring of polynomials or Laurent polynomials (with one or more variables)
over the integers. Changing the ground rings to fields, satisfying some mild
assumptions, semisimplicity criteria are obtained that imply semisimplicity
in an open set of parameter values. As semisimplicity means non-vanishing
of all Gram determinants, it follows that the factors in the formulae do
not vanish. Therefore, determinants of swich matrices generically are
non-zero and thus not zero-divisors over (Laurent) polynomial rings or
over fields. There are however explicit examples where some determinants
of swich matrices do vanish.

\bigskip

{\bf Acknowledgements.}

This work is the result of a long-standing collaboration between the
authors. The first and last authors would like to express their sincere
gratitude to the Institute for Algebra and Number Theory, in particular to
Steffen Koenig, his students and postdoctoral researchers, for their warm
hospitality, creating an optimal environment to carry out research.
Paula Carvalho and Christian Lomp were partially supported by CMUP – Centro
de Matem\'{a}tica da Universidade do Porto, member of LASI, which is financed by
national funds through FCT – Funda\c{c}\~{a}o para a
Ci\^encia e a Tecnologia, I.P.,
under the project with reference UID/00144/2025, doi:
https://doi.org/10.54499/UID/00144/2025.


\bigskip

\begin{bibdiv}
  \begin{biblist}
    \bib{Atiyah}{book}{
      AUTHOR={Atiyah, Michael Francis},
      Author={Macdonald, Ian Grant},
     TITLE = {Introduction to commutative algebra},
 PUBLISHER = {Addison-Wesley Publishing Co., Reading, Mass.-London-Don
              Mills, Ont.},
      YEAR = {1969},
     PAGES = {ix+128},
 }
    
 \bib{BN}{article}{
   author={Baum, Paul},
   author={Nistor, Victor},
   title={Periodic cyclic homology of Iwahori-Hecke algebras},
   journal={$K$-Theory},
   volume={27},
   date={2002},
   number={4},
   pages={329--357},
   issn={0920-3036},
}

\bib{Brown}{book}{
   author={Brown, William C.},
   title={Matrices over commutative rings},
   series={Monographs and Textbooks in Pure and Applied Mathematics},
   volume={169},
   publisher={Marcel Dekker, Inc., New York},
   date={1993},
   pages={viii+281},
   isbn={0-8247-8755-2},
}

\bib{Brown1955}{article}{
 AUTHOR = {Brown, William Price},
     TITLE = {Generalized matrix algebras},
   JOURNAL = {Canadian J. Math.},
    VOLUME = {7},
      YEAR = {1955},
     PAGES = {188--190},
}
     
\bib{CKLS}{article}{
   author={Carvalho, Paula A. A. B.},
   author={Koenig, Steffen},
   author={Lomp, Christian},
   author={Shalile, Armin},
   title={Ring theoretical properties of affine cellular algebras},
   journal={J. Algebra},
   volume={476},
   date={2017},
   pages={494--518},
   issn={0021-8693},
}

\bib{GeckJacon}{book}{
   author={Geck, Meinolf},
   author={Jacon, Nicolas},
   title={Representations of Hecke algebras at roots of unity},
   series={Algebra and Applications},
   volume={15},
   publisher={Springer-Verlag London, Ltd., London},
   date={2011},
   pages={xii+401},
   isbn={978-0-85729-715-0},
}

\bib{GrahamLehrer}{article}{
   author={Graham, John J.},
   author={Lehrer, Gustav I.},
   title={Cellular algebras},
   journal={Invent. Math.},
   volume={123},
   date={1996},
   number={1},
   pages={1--34},
   issn={0020-9910},
 }

 \bib{Green}{book}{
   author={Green, James Alexander},
   title={Polynomial representations of $GL_n$},
   series={Lecture Notes in Mathematics},
   volume={830},
   date={1980},
   publisher={Springer-Verlag},
   }

\bib{HuMathas}{article}{
    AUTHOR = {Hu, Jun}, author={Mathas, Andrew},
     TITLE = {Seminormal forms and cyclotomic quiver {H}ecke algebras of
              type {$A$}},
   JOURNAL = {Math. Ann.},
    VOLUME = {364},
      YEAR = {2016},
    NUMBER = {3-4},
    PAGES = {1189--1254},
    }

\bib {humphreys}{book}{
    AUTHOR = {Humphreys, James E.},
     TITLE = {Reflection groups and {C}oxeter groups},
    SERIES = {Cambridge Studies in Advanced Mathematics},
    VOLUME = {29},
 PUBLISHER = {Cambridge University Press, Cambridge},
      YEAR = {1990},
     PAGES = {xii+204},
}
    
\bib{KazhdanLusztig}{article}{
   author={Kazhdan, David},
   author={Lusztig, George},
   title={Representations of Coxeter groups and Hecke algebras},
   journal={Invent. Math.},
   volume={53},
   date={1979},
   number={2},
   pages={165--184},
   issn={0020-9910},
}

\bib{Kleshchev}{article}{
AUTHOR = {Kleshchev, Alexander S.},
     TITLE = {Affine highest weight categories and affine quasihereditary
              algebras},
   JOURNAL = {Proc. Lond. Math. Soc. (3)},
    VOLUME = {110},
      YEAR = {2015},
    NUMBER = {4},
    PAGES = {841--882},
    }

\bib{KLM}{article}{
   author={Kleshchev, Alexander S.},
   author={Loubert, Joseph W.},
   author={Miemietz, Vanessa},
   title={Affine cellularity of Khovanov-Lauda-Rouquier algebras in type
   $A$},
   journal={J. Lond. Math. Soc. (2)},
   volume={88},
   date={2013},
   number={2},
   pages={338--358},
   issn={0024-6107},
}

\bib{KL}{article}{
   author={Kleshchev, Alexander S.},
   author={Loubert, Joseph W.},
   title={Affine cellularity of Khovanov-Lauda-Rouquier algebras of finite
   types},
   journal={Int. Math. Res. Not. IMRN},
   date={2015},
   number={14},
   pages={5659--5709},
   issn={1073-7928},
}
 
 \bib{KS}{article}{
   author={Kleshchev, Alexander S.},
   author={Steinberg, David J.},
   title={Homomorphisms between standard modules over finite-type KLR
   algebras},
   journal={Compos. Math.},
   volume={153},
   date={2017},
   number={3},
   pages={621--646},
   issn={0010-437X},
}

\bib{KXcellular}{article}{
   author={Koenig, Steffen},
   author={Xi, Changchang},
TITLE = {On the structure of cellular algebras},
 BOOKTITLE = {Algebras and modules, {II} ({G}eiranger, 1996)},
    SERIES = {CMS Conf. Proc.},
    VOLUME = {24},
     PAGES = {365--386},
 PUBLISHER = {Amer. Math. Soc., Providence, RI},
 YEAR = {1998},
 }

\bib{KXwhenis}{article}{
   author={Koenig, Steffen},
   author={Xi, Changchang},
       TITLE = {When is a cellular algebra quasi-hereditary?},
   JOURNAL = {Math. Ann.},
    VOLUME = {315},
      YEAR = {1999},
    NUMBER = {2},
    PAGES = {281--293},
    }

\bib{KXaffine}{article}{
   author={Koenig, Steffen},
   author={Xi, Changchang},
   title={Affine cellular algebras},
   journal={Adv. Math.},
   volume={229},
   date={2012},
   number={1},
   pages={139--182},
   issn={0001-8708},
}

\bib{Lusztig2}{article}{
	author={Lusztig, George},
	title = {Cells in affine Weyl groups. II.},
	journal = {J. Algebra},
	volume={109}, 
	pages={536--548},
	date={1987}
	}

\bib{Mathas}{article}{
   author={Mathas, Andrew},
   title={Seminormal forms and Gram determinants for cellular algebras},
   note={With an appendix by Marcos Soriano},
   journal={J. Reine Angew. Math.},
   volume={619},
   date={2008},
   pages={141--173},
   issn={0075-4102},
}

\bib{Mathas99}{book}{
   author={Mathas, Andrew},
   title={Iwahori-Hecke algebras and Schur algebras of the symmetric group},
   series={University Lecture Series},
   volume={15},
   publisher={American Mathematical Society, Providence, RI},
   date={1999},
   pages={xiv+188},
   isbn={0-8218-1926-7},
}

\bib{McCR}{book}{
   author={McConnell, John C.},
   author={Robson, J. Chris},
   title={Noncommutative Noetherian rings},
   series={Graduate Studies in Mathematics},
   volume={30},
   edition={Revised edition},
   note={With the cooperation of L. W. Small},
   publisher={American Mathematical Society, Providence, RI},
   date={2001},
   pages={xx+636},
   isbn={0-8218-2169-5},
 }

 \bib{RuiSibrauer}{article}{
 AUTHOR = {Rui, Hebing}, author={Si, Mei},
     TITLE = {Discriminants of {B}rauer algebras},
   JOURNAL = {Math. Z.},
    VOLUME = {258},
      YEAR = {2008},
    NUMBER = {4},
    PAGES = {925--944},
    }

 \bib{RuiSibmw}{article}{
 AUTHOR = {Rui, Hebing}, author={Si, Mei},
     TITLE = {Gram determinants and semisimplicity criteria for
              {B}irman-{W}enzl algebras},
   JOURNAL = {J. Reine Angew. Math.},
     VOLUME = {631},
      YEAR = {2009},
     PAGES = {153--179},
}
     
\bib{Xi}{article}{
   author={Xi, Nanhua},
   title={The based ring of two-sided cells of affine Weyl groups of type
   $\widetilde A_{n-1}$},
   journal={Mem. Amer. Math. Soc.},
   volume={157},
   date={2002},
   number={749},
   pages={xiv+95},
   issn={0065-9266},
}

 \end{biblist}
\end{bibdiv}
\end{document}